\documentclass{amsart}

\usepackage[dvips]{graphicx}
\usepackage{xcolor}

\theoremstyle{plain}
\newtheorem{theorem}{Theorem}[section]
\newtheorem*{theorem*}{Theorem}
\newtheorem{lemma}[theorem]{Lemma}

\theoremstyle{definition}
\newtheorem{definition}[theorem]{Definition}

\theoremstyle{remark}
\newtheorem{remark}[theorem]{Remark}
\newtheorem*{acknowledgements}{Acknowledgements}

\newcommand{\del}{\partial}
\newcommand{\R}{\mathbb{R}}
\newcommand{\N}{\mathbb{N}}
\renewcommand{\H}{\mathbb{H}}
\newcommand{\C}{\mathbb{C}}

\newcommand{\E}{\mathbb{E}}

\newcommand{\lift}{\widetilde}
\DeclareMathOperator{\instar}{instar}
\DeclareMathOperator{\arcsinh}{arcsinh}
\DeclareMathOperator{\arctanh}{arctanh}

\begin{document}

\title{Bounds on Pachner moves and systoles of cusped 3-manifolds}

\author[Kalelkar]{Tejas Kalelkar}
\address{Mathematics Department, Indian Institute of Science Education and Research, Pune 411008, India}
\email{tejas@iiserpune.ac.in}

\author[Raghunath]{Sriram Raghunath}
\address{Mathematics Department, Indian Institute of Science Education and Research, Pune 411008, India}
\email{sriram.raghunath@students.iiserpune.ac.in}

\date{\today}

\keywords{Hauptvermutung, ideal triangulations, hyperbolic knots, Pachner moves, systole length}

\subjclass[2010]{Primary 57M25, 57Q25}

\begin{abstract}
Any two geometric ideal triangulations of a cusped complete hyperbolic $3$-manifold $M$ are related by a sequence of Pachner moves through topological triangulations. We give a bound on the length of this sequence in terms of the total number of tetrahedra and a lower bound on dihedral angles. This leads to a naive but effective algorithm to check if two hyperbolic knots are equivalent, given geometric ideal triangulations of their complements. Given a geometric ideal triangulation of $M$, we also give a lower bound on the systole length of $M$ in terms of the number of tetrahedra and a lower bound on dihedral angles.
\end{abstract}

\maketitle

\section{Introduction}
A basic question in knot theory is to determine when two given knots or links are equivalent. There are several algorithms to solve this equivalence problem but the complexity class of this problem is still not known. 

Haken\cite{Hak} gave an algorithm for non-fibered knots in the sixties using a hierarchy of normal surfaces that gives a canonical cell structure on the knot complement. Equivalence then follows from a result by Gordon and Luecke\cite{GorLue} which shows that a knot is determined by the homeomorphism class of its complement (up to mirror images). Haken's algorithm was extended to fibered knots using a solution of the conjugacy class problem for mapping class groups by Hemion\cite{Hem}. A complete rigorous treatment was recently given by Matveev\cite{Mat}. We are not aware of any explicit estimations of the complexity of Haken-Hemion-Matveev's algorithm.

Thurston has classified nontrivial knots in $S^3$ into torus, satellite and hyperbolic knots. Hyperbolic knots are those whose complements in $S^3$ are complete orientable one-cusped hyperbolic $3$-manifolds. Generically, knots with small crossing numbers\cite{HosThiWee}, alternating diagrams\cite{Men} or highly twisted diagrams\cite{FutPur} are hyperbolic. Given a triangulation of a hyperbolic knot complement, we can calculate a presentation of its fundamental group, which is Kleinian. Sela\cite{Sel} has given an algorithm to solve the isomorphism problem for Kleinian groups. By the rigidity result of Mostow-Prasad\cite{Mos}\cite{Pra}, the fundamental group of a complete hyperbolic $3$-manifold determines the manifold up to isometry. Combining these results gives an algorithm for equivalence of hyperbolic knots. Sela's algorithm though is of an existential nature and this procedure does not lead to a practical algorithm with explicit computation bounds.

An algorithm given by Casson-Manning-Weeks\cite{Man, Wee} involves computing a canonical ideal polyhedral decomposition of the hyperbolic knot complement, called the Epstein-Penner decomposition\cite{EpsPen}. Commonly used software to study hyperbolic manifolds like SnapPea attempt to implement their algorithm to compute this decomposition. In practice, this seems to be the most efficient way to recognise hyperbolic knots. However, unlike the Casson-Manning algorithm SnapPea is not guaranteed to always find the Epstein-Penner decomposition. There may also creep in floating point errors which can lead to equivalent manifolds being considered distinct\cite{Bur}.

More recently, Kuperberg\cite{Kup} has given a complete proof of the folklore result from the 1970s that the homeomorphism problem for closed oriented $3$-manifolds is a corollary of geometrization. An algorithm is said to be elementary recursive if its execution time is bounded by a bounded tower of exponentials. In his paper it is shown that the complexity class of the oriented homeomorphism problem for closed oriented triangulated $3$-manifolds is elementary recursive. Furthermore, for cusped  hyperbolic $3$-manifolds it is shown (Theorem 8.3 of \cite{Kup}) that it is elementary recursive to find a geometric triangulation and specify its geometric data with algebraic numbers. And consequently, the isomorphism problem for cusped hyperbolic $3$-manifolds is also elementary recursive. In this article, we assume that we are already given geometric ideal triangulations of two cusped hyperbolic $3$-manifolds along with its geometric data (lower bound on dihedral angles), and we then proceed to give a doubly exponential time algorithm for the isomorphism problem.

Algorithms with explicit computation bounds have been calculated using either Reidemeister or Pachner moves. Reidemeister moves are local changes to the diagram of the knot while bistellar or Pachner moves are local changes to a triangulation of the knot complement. There are only 3 pairs of Reidemeister and 2 pairs of Pachner moves, so an explicit bound on the number of moves needed to relate knot diagrams or triangulations of knot complements leads to an algorithm to solve the knot equivalence problem with explicit running time bounds. Coward and Lackenby\cite{CowLac} have given such a bound for Reidemeister moves while Mijatovic\cite{Mij4} has given such a bound for Pachner moves. These bounds though are huge. The bound on Reidemeister moves is a tower of exponentials of  height $10^{1000000m}$, where $m$ is the crossing number of the diagram. The bound on Pachner moves is a tower of exponentials of height $2^{200m}$, where $m$ is the number of tetrahedra in the triangulation of the knot complement.

It is conjectured that hyperbolic knot complements always have geometric ideal triangulations. Ham and Purcell\cite{HamPur} have recently proved this for sufficiently highly twisted knots while Luo, Schleimer and Tillmann\cite{LuoSchTil} have proved that such triangulations always exist virtually.  Geometric triangulations of a cusped $3$-manifold $M$ may not be unique, for example the complement of the Figure Eight in $S^3$ is a complete orientable one-cusped hyperbolic $3$-manifold with infinitely many geometric ideal triangulations \cite{DadDua}. Any two topological ideal  triangulations of $M$ are related by a sequence of Pachner moves through topological ideal triangulations\cite{Ame}. It is remarked in \cite{DadDua} that the Figure Eight knot complement has geometric ideal triangulations which can not be related by Pachner moves through geometric ideal triangulations.

It is natural then to ask if there are better bounds for the knot equivalence problem given geometric ideal triangulations instead of topological triangulations of the knot complement. The aim of this article is obtain a substantially lower explicit bound on the number of Pachner moves needed to relate geometric ideal triangulations of cusped hyperbolic manifolds, when the intermediate triangulations are allowed to be topological (not geometric) and have material (non-ideal) vertices. 

\begin{theorem}\label{mainthm1}
Let $M$ be a complete orientable cusped hyperbolic $3$-manifold. Let $\tau_1$ and $\tau_2$ be geometric ideal triangulations of $M$ with at most $m_1$ and $m_2$ many tetrahedra respectively and all dihedral angles at least $\theta_0$. Let $m=m_1+ m_2$. Then the number of Pachner moves needed to relate $\tau_1$ and $\tau_2$ is less than
$$ (2.8\times 10^{12})\cdot \frac{ m^{11/2}}{(\sin \theta_0)^{12m+27/2}} $$
\end{theorem}

This leads to a conceptually simple algorithm to solve the hyperbolic knot equivalence problem with explicit running time bounds:\\
\emph{Hyperbolic knot equivalence algorithm} Let $\kappa_1$ and $\kappa_2$ be two hyperbolic knots in $S^3$. Let $\tau_1$ and $\tau_2$ be geometric ideal triangulations of their complements with $m_1$ and $m_2$ many tetrahedra and all dihedral angles at least $\theta_0$. Let $m=m_1 + m_2$. The algorithm proceeds as follows: Make a list $\mathcal{L}$ of all triangulations that are less than $N(m, \theta_0)$ Pachner moves away from $\tau_1$ where $N(m, \theta_0)$ is the upper bound calculated in Theorem \ref{mainthm1}. There are only 4 possible Pachner moves so this is a finite constructible list of triangulations. We will argue that $\kappa_1$ is equivalent to $\kappa_2$ if and only if some triangulation in $\mathcal{L}$ is combinatiorially isomorphic to $\tau_2$. 

If we can find such a combinatorial isomorphism then the given knot complements are homeomorphic and hence by Gordon-Luecke\cite{GorLue} the two knots are equivalent (up to mirror images). Conversely if the two knots are equivalent, then there exists a homeomorphism between their complements. By the Mostow-Prasad rigidity\cite{Mos}\cite{Pra}, we may assume this homeomorphism is in fact an isometry $h:S^3 \setminus \kappa_1 \to S^3 \setminus \kappa_2$. By Theorem \ref{mainthm1}, $\tau_1$ and $h^{-1}(\tau_2)$ are related by less than  $N(m, \theta_0)$ Pachner moves. Therefore $h$ is a combinatorial isomorphism between $h^{-1}(\tau_2)$ and $\tau_2$, where $h^{-1}(\tau_2)$ is in the list $\mathcal{L}$.\\

Finite element methods which use geometric triangulations often assume that there are no slivers, i.e., tetrahedra with very small dihedral angles. We call a geometric ideal triangulation $\theta_0$-thick if all its dihedral angles are at least $\theta_0$. In Section 2 we use the Euclidean triangulation induced on the cusp tori of $M$ to  calculate a minimum distance between the edges of a $\theta_0$-thick triangulation in the thick part of $M$. This allows us to give a bound on the number of polytopes in $\tau_1 \cap \tau_2$. The manifold $M$ is non-compact, so it is not a priori obvious why such a bound should even exist. Our required theorem then follows in Section 3 from  previous work by Phanse and the first author\cite{KalPha} which gives a bound  on the number of Pachner moves needed to relate a geometric triangulation and its geometric subdivision. In that paper a bound is calculated on the number of Pachner moves needed to relate geometric triangulations of compact constant-curvature $n$-manifolds, in terms of an upper bound on the length of edges and number of tetrahedra.

The existence of a common geometric subdivision also allows us to prove that any two geometric ideal triangulations are in fact related by Pachner moves through geometric triangulations (possibly with material vertices).

\begin{theorem}\label{mainthm2}
Let $M$ be a complete orientable cusped hyperbolic $3$-manifold. Any two geometric ideal triangulations of $M$ are related by a sequence of Pachner moves through geometric (possibly non-ideal) triangulations.
\end{theorem}

It is tempting to try and prove this using a simplicial cobordism between the given geometric triangulations $\tau_0$ and $\tau_1$ of $M$, as done for convex polytopes in $\R^n$ by Izmestiev and Schlenker\cite{IzmSch}. This would entail putting a geometry on $M \times I$ that agrees with the hyperbolic structure of $M$ on $M \times \{0\}$ and $M \times \{1\}$. Then extending the triangulations $\tau_0$ of $M \times \{0\}$ and $\tau_1$ of $M \times \{1\}$ to a geometric triangulation of $M \times I$. The sequence of Pachner move would then be obtained by inductively removing $(n+1)$-dimensional simplexes of this triangulation from above and projecting the triangulation of the upper boundary onto $M \times 0$. However, $M \times I$ may not admit any geometric triangulation with the given boundary constraints. And furthermore, there may not be a topmost $(n+1)$-simplex such that the projection map restricts injectively on the upper boundary and takes geometric simplexes to geometric simplexes. These hurdles are handled in \cite{KalPha2} by working locally in stars of simplexes and using the property of regularity of a suitable subdivision.\\

Breslin \cite{Bre} has shown that there exists a constant $L$ such that every complete hyperbolic $3$-manifold $M$ has a geometric triangulation $\tau$ (with material vertices) such that every tetrahedron of $\tau$ that lies in the thick part of $M$ is $L$-bilipschitz diffeomorphic to the standard Euclidean tetrahedron. In contrast, we observe below that for any $\theta_0>0$ there exist knots whose complements have no $\theta_0$-thick geometric ideal triangulations. 

\begin{remark}
Let $\mathcal{F}(V, \theta_0)$ be the family of hyperbolic knots $\kappa$ in $S^3$ such that $S^3 \setminus \kappa$ admits a $\theta_0$-thick geometric ideal triangulation $\tau_\kappa$ and has volume less than $V$. Hyperbolic ideal tetrahedra are completely determined by their dihedral angles $(\alpha, \beta, \gamma)$ such that $\alpha + \beta + \gamma=\pi$. And the volume functional is continuous on the compact set of ideal hyperbolic tetrahedra $\{(\alpha, \beta, \gamma) | \alpha + \beta + \gamma=\pi;\; \alpha, \beta, \gamma \geq \theta_0\}$. Therefore the volume of a tetrahedron in any $\theta_0$-thick geometric ideal triangulation is at least some positive number $v$. This minimum is non-zero, as the volume of a tetrahedron is zero only if one of its dihedral angles is zero. If some  knot $\kappa$ in $\mathcal{F}(V, \theta_0)$ admitted a geometric ideal $\theta_0$-thick triangulation $\tau_\kappa$ with more than $V/v$ tetrahedra then the volume of $S^3 \setminus \kappa$ would be greater than $(V/v) v=V$, which is a contradiction. So for any knot $\kappa \in \mathcal{F}(V, \theta_0)$ the triangulation $\tau_\kappa$ has less than $V/v$ tetrahedra. But there are only finitely many topological manifolds that can be constructed with less than $V/v$ tetrahedra. By Gordon-Luecke\cite{GorLue}, knots in $S^3$ are determined by their complements and so $\mathcal{F}(V, \theta_0)$ has only finitely many knots.

Every prime knot has a prime, twist-reduced diagram. Lackenby\cite{Lac} has shown that the volume of a hyperbolic knot with a prime twist-reduced diagram $D$ is bounded above by $V=10\, v_{tet}(t(D)-1)$ where $v_{tet}$ is the volume of the ideal regular hyperbolic tetrahedron and $t(D)$ is the number of twist regions in $D$. See Figure \ref{twists} for an example of a twist region with $6$ crossings. Futer and Purcell\cite{FutPur} have shown that if the number of crossings in each twist region of $D$ is more than $6$ then the knot is hyperbolic. So an easy way to obtain infinitely many hyperbolic knots with volume less than $V$ is by repeatedly twisting a pair of strands in a twist region of $D$.  Only finitely many of these knots have complements that can admit $\theta_0$-thick triangulations.
\end{remark}

\begin{figure}
\centering
\def\svgwidth{0.2\columnwidth}
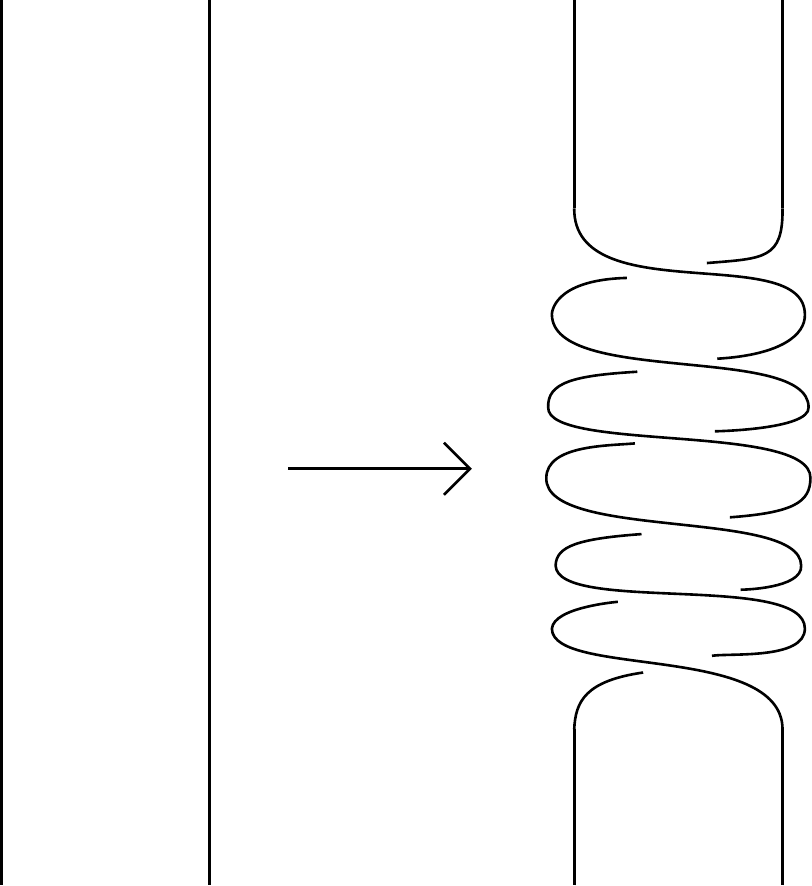
\caption{A twist region of a knot diagram with 6 crossings}\label{twists}
\end{figure}

The systole length of a hyperbolic knot is the length of a shortest closed geodesic in the knot complement. Every hyperbolic knot complement contains a simple closed geodesic\cite{AdaHasSco} so the systole length of a hyperbolic knot is an important property of the knot. In Section 2 we calculate a lower bound on the systole length of $M$ in terms of the number of tetrahedra and a lower dihedral angle bound of a geometric ideal triangulation of $M$.

\begin{theorem}\label{mainthm3}
Let $M$ be a complete orientable cusped hyperbolic $3$-manifold. Let $\tau$ be a geometric ideal triangulation of $M$ with $m$ many tetrahedra and all dihedral angles at least $\theta_0$. The systole length of $M$ is bounded below by 
$$
2^{-9} \cdot  \frac{(\sin\theta_0)^{4m+7/2}}{m^{3/2}}
$$
\end{theorem}
 
Note that this lower bound is far from ideal. For example the Figure Eight knot complement is triangulated by $2$ regular hyperbolic ideal tetrahedra. So the lower bound for its systole length obtained from Theorem \ref{mainthm3} is $1.32 \times 10^{-4}$ which is far from its actual systole length of $1.087$\cite{AdaRei}.

\section{Systoles and a common subdivision with boundedly many polytopes}
The main idea used extensively in this section is the notion of a thick-thin decomposition for cusped complete hyperbolic $3$-manifolds, which is a consequence of Margulis' Lemma (see Theorem D.3.13 of \cite{BenPet} for a modern treatment). Let $M$ be a complete orientable cusped hyperbolic $3$-manifold. For $\epsilon>0$, let $M^{\epsilon}$ denote the set of points of $M$ with injectivity radius less than $\epsilon/2$.  Margulis showed that there exists a universal constant $\epsilon>0$ such that $M^\epsilon$ consists of tubes around closed geodesics of length less than $\epsilon$ and neighbourhoods $C_i$ of the cusps each of which is homeomorphic to $T \times (0,\infty)$, where $T$ is a torus. $M^\epsilon$ is known as  the thin part of $M$ and its complement is called the thick part. For non-compact orientable hyperbolic $3$-manifolds, $0.29\leq \epsilon <1$. The lower bound for $\epsilon$ is due to Shalen\cite{Sha} and the upper bound follows from the fact that the optimal Margulis number for the Figure-8 knot complement is less than $1$ (see \cite{Ada2002}).

In the rest of this article, we fix the upper half-space model $\H^3$ for the hyperbolic 3-space, we identify its ideal boundary $\del \H^3$ with $\C \cup \infty$ and the group of orientation preserving isometries of $\H^3$ with $PSL(2, \C)$. We use the notation $H(h)$ to denote the horoball $\{(x, y, z) \in \H^3: z>h\}$ and $X_z = \R^2 \times \{z\}$ to denote the boundary of the closure of $H(h)$ in $\H^3$. We give an outline of the results in this section below:
\newline

Let $\pi: \H^3 \to M$ be a covering map, with $\infty$ sent to a cusp $c$ of $M$. Let $\Gamma_\infty$ be the group of deck transformations of $\pi$ that fix $\infty$. The group $\Gamma_\infty$ is generated by a pair of parabolic isometries of $\H^3$. Let $C$ denote the cusp neighbourhood of $c$ in $M$ with points of injectivity radius less than $\epsilon/2$, i.e., $C$ is the component of $M^\epsilon$ which is a neighbourhood of $c$. The structure of the thick-thin decomposition tells us that there exists $h_0>0$ such that $\pi: H(h_0) \to H(h_0)/\pi_1(M) = H(h_0)/\Gamma_\infty = C$. For any $z>h_0$, we say that $T_z = X_z/\Gamma_\infty=\pi(X_z)$ is a cusp torus embedded in $C$ at height $z$. This height $z$ of the cusp torus $T_z$ clearly depends on the choice of the covering map $\pi$. 

In the first part of this section, we introduce the idea of normalising a covering projection. Let $T_0$ denote the flat torus $\R^2 \times \{0\}/\Gamma_\infty$. By composing the given covering map $\pi$ with a suitable hyperbolic isometry we show  that $\pi$  can be normalised (Lemma \ref{normalised_existence}), in the sense that the corresponding flat torus $T_0$ has shortest Euclidean closed geodesic of unit length. Any two covering projections which are normalised with respect to the same cusp are related by an isometry of $\H^3$ that preserves horoballs $H(h)$ centered at $\infty$ (Lemma \ref{normalised_uniqueness}). We can therefore define the normalised height of a cusp torus $T_z$ in $C$ as the height of the cusp torus with respect to any covering projection which is normalised with respect to $c$. We denote by $h_0(c)$ the normalised height of the cusp torus bounding $C$. The hyperbolic length of a shortest closed geodesic in $T_z \subset C$ is $1/z$, so we can show that $1/\epsilon$ is an upper bound for $h_0(c)$ (Lemma \ref{cuspbounds}). The volume of the cusp neighbourhood $C$ bounded by $T_{h_0(c)}$ is less than the volume $V$ of the manifold, so we can also calculate an upper bound for the Euclidean area $A_0(c)$ of the flat cusp torus $T_0$ in terms of $\epsilon$ and $V$ (Lemma \ref{cuspbounds}). 

Let $\bar{\tau}$ be a Euclidean triangulation of a flat torus all of whose angles are at least $\theta_0$. We next find lower and upper bounds on the lengths of edges of such a triangulation $\bar{\tau}$ in terms of upper bounds on the number of triangles in $\bar{\tau}$ and the area of the torus (Lemma \ref{torusedgebounds}). We also calculate an upper bound on the circumradius for the set of Euclidean triangles of bounded area, with angles at least $\theta_0$ (Lemma \ref{circumrad}). Each ideal hyperbolic triangle which is not vertical lies on a hemisphere with center on the plane $\R^2\times \{0\}$ whose radius is the cirucmradius of the Euclidean triangle in $\R^2 \times \{0\}$ with the same vertices. An ideal triangulation $\tau$ of $M$ induces a Euclidean triangulation $\bar{\tau}$ on a cusp torus $T_z \subset C$ exactly when $T_z$ intersects only those faces of $\tau$ which have $c$ as one of its vertices. We need to take $z$ high enough so as to avoid all those faces of $\tau$ which do not intersect $T_z$ orthogonally. The upper bound on the circumradii of the triangles in $\bar{\tau}$ allows us to calculate such a height $z_0\geq h_0(c)$ (Lemma \ref{taunormalised_existence}). 

Let $M(z_0)$ be the compact hyperbolic manifold obtained from $M$ by removing all the cusp neighbourhoods bounded by the cusp tori $T_{z_0}$ at height $z_0$ in the various cusps. The boundary of $M(z_0)$ is a disjoint union of flat tori and the geometric ideal triangulation $\tau$ of $M$ induces a Euclidean triangulation on them. In the second part of this section, we obtain a lower bound on the distance between the edges of $\tau \cap M(z_0)$ in $M$ using the lower bound we have calculated for the Euclidean edge lengths of the induced triangulations on $\del M(z_0)$ (Lemma \ref{edgeball}). Similarly, for a point $p$ in a 2-cell of $\tau \cap M(z_0)$ that is sufficiently far from the boundary of the face, we calculate a lower bound on the distance between $p$ and the other faces of $\tau$ (Lemma \ref{faceball}). These two bounds together give a lower bound on the `thickness' of stars of cells in $\tau \cap M(z_0)$. In particular, we obtain a lower bound for the injectivity radius of points in the $2$-cells of $\tau \cap M(z_0)$. Every closed geodesic of $M$ intersects a face of $\tau$ at some point in $M(z_0)$, which therefore gives us a lower bound on the systole length of $M$ (Theorem \ref{mainthm3}).

The aim of the third part of this section is to obtain an upper bound on the number of polytopes in the polytopal complex $\tau_1 \cap \tau_2$ obtained by the intersection of two ideal geometric triangulations $\tau_1$ and $\tau_2$ of $M$. We first obtain an upper bound on the number of components in the intersection of the edges of $\tau_1$ with a fixed tetrahedron $\Delta$ of $\tau_2$. For each such component $\sigma$, we show using the bounds calculated for thickness of stars, that there is a $\theta_0$-sector of a ball of radius $r_0$ which lies entirely in $\Delta$ (Lemma \ref{sectorvol}). Furthermore all such sectors for different components $\sigma$ are pairwise disjoint. The volume of $\Delta$ is bounded above by the volume of the regular ideal tetrahedron, so we can calculate a bound on the number of such components (Lemma \ref{finiteint}). We then show that all polytopes (except perhaps one) in the polytopal complex $\Delta \cap \tau_1$ must have an edge that is part of an edge of $\Delta$ or an edge of $\tau_1$ (Lemma \ref{edgesub}). Combining these results we get a bound on the number of polytopes in $\tau_1 \cap \tau_2$ (Theorem \ref{mainsubthm}). Once we have shown finiteness of this intersection it follows, using a result of \cite{KalPha2}, that any two geometric  ideal triangulations of $M$ are related by Pachner moves through geometric triangulations that may be non-ideal (Theorem \ref{mainthm2}).

\subsection{Normalised covering maps}
Let $c$ be a cusp of $M$ and let $\tau$ be an ideal triangulation of $M$. In this section we introduce the notion of normalising a covering projection with respect to $c$. We show that normalised cusp heights are independent of the choice of the normalised covering projections with respect to the cusp. We obtain upper bounds on the normalised cusp height, the $\tau$-normalised cusp height and the area of the normalised flat cusp torus of $M$. We also calculate a lower bound on the Euclidean length of edges of a $\theta_0$-thick Euclidean triangulation of a flat torus of bounded area. We list out all these bounds at the end of this subsection in Remark \ref{constants}.

\begin{definition}\label{normalised_defn}
Let $pr: \H^3 \to M$ be a covering map that sends $\infty$ to the cusp $c$ of $M$. Let $\Gamma_\infty$ be the subgroup of the group of deck transformations of $pr$ which fix $\infty$. We say that $pr$ is normalised with respect to $c$ if $\R^2 \times \{0\}/\Gamma_\infty$ is a flat torus with shortest closed Euclidean geodesic of length $1$.
\end{definition}

\begin{definition}\label{normalised_notation}
Let $\pi: \H^3 \to M$ be a covering projection that sends $\infty$ to a cusp $c$ of $M$. We fix some notations related to consequences of the thick-thin decomposition of $M$:
\begin{enumerate}
\item{Let $\Gamma_\infty$ be the subgroup of the group of deck transformations of $\pi$ which fix $\infty$. Then $\Gamma_\infty$ is generated by a pair of independent parabolic isometries $\alpha$ and $\beta$, i.e., isometries of the kind $\alpha(p) = p+u$ and $\beta(p)= p + v$ with $u$ and $v$ linearly independent vectors of $\R^2 \times \{0\}$. We call this presentation of the cusp group $\Gamma_\infty$ the presentation with respect to $\pi$.}
\item{Let $H(h)$ be the horoball $\{(x, y, z) \in \H^3: z>h\}$. Let $C$ be the cusp neighbourhood of the cusp $c$, which has injectivity radius less than $\epsilon/2$ at every point. There exists $h>0$ such that $\pi$ induces an isometry from the set $H(h)/\Gamma_\infty$ to $C$. We call such an $h$ the cusp height with respect to $\pi$.
When $\pi$ is normalised, we denote this cusp height by $h_0(c)$ and call it the normalised cusp height of $c$. We shall show in Lemma \ref{normalised_uniqueness} that the normalised cusp height $h_0(c)$ does not depend on the choice of the normalised covering map with respect to $c$.}
\item{Let $P(u, v)$ denote the Euclidean parallelogram spanned by $u$ and $v$ in $\R^2 \times \{0\}$. Let $T_0$ denote the flat torus $\R^2 \times \{0\}/\Gamma_\infty = P(u, v)/<\alpha, \beta>$. We call $T_0$ the flat cusp torus with respect to $\pi$. By definition when $\pi$ is normalised, any shortest closed Euclidean geodesic in $T_0$ is of unit length. We call the Euclidean area of such a $T_0$ the area of the normalised flat cusp torus and denote it by $A_0(c)$. We shall show in Lemma \ref{normalised_uniqueness} that the area of the normalised flat cusp torus $A_0(c)$ does not depend on the choice of the normalised covering map with respect to $c$.}
\end{enumerate}
\end{definition}

We next show that any covering projection can be normalised by composing with a suitable hyperbolic isometry.
\begin{lemma}\label{normalised_existence}
Let $M$ be a complete cusped orientable hyperbolic $3$-manifold. Let $\pi: \H^3 \to M$ be a covering projection sending $\infty$ to a cusp $c$ of $M$. Then there exists a hyperbolic isometry $\psi$ of $\H^3$ fixing $\infty$ such that $pr= \pi \circ \psi$ is a covering projection that is normalised with respect to the cusp $c$.
\end{lemma}
\begin{proof}
As in the notation fixed in Definition \ref{normalised_notation}, let $\Gamma_\infty$ be the subgroup of deck transformations of $\pi$ which fix $\infty$. It is generated by two parabolic elements $\alpha$ and $\beta$ which have the form $\alpha(p) =p + u$ and $\beta(p)=p + v$, with $u$ and $v$ linearly independent vectors of $\R^2 \times \{0\}$. Let $P(u, v)$ denote the Euclidean parallelogram spanned by $u$ and $v$ in $\R^2 \times \{0\}$. Let $l$ be the length of a shortest closed Euclidean geodesic in the flat torus $T_0=P(u, v)/\Gamma_\infty$. Let $\phi(z) = (1/l) z$ be a hyperbolic isometry of $\H^3$. Let $h$ denote the cusp height with respect to $\pi$ and let $h_0=h/l$. Let $\Gamma'_\infty = \phi \circ \Gamma_\infty \circ \phi^{-1}$. Let $P'$ be the parallelogram $\phi(P(u, v))$ i.e., $P(u,v)$ scaled by $1/l$. Then $\phi$ takes the horoball $H(h)$ to the horoball $H(h_0)$, the quotient space $H(h)/\Gamma_\infty$ to $H(h_0)/\Gamma'_\infty$ and the flat torus $T_0=P(u, v)/\Gamma_\infty$ to the flat torus $T'_0=P'/\Gamma'_\infty$.

Let $\psi = \phi^{-1}$ and let $pr: \H^3 \to M$ be the covering map $pr=\pi \circ \psi$. Then $\Gamma'_\infty$ is the subgroup of the group of deck transformations of $pr$ that fix $\infty$. Also $pr$ induces an isometry from $C(h_0)=H(h_0)/\Gamma'_\infty$ to the cusp neighbourhood $C$ of $c$.  And finally by construction a shortest closed Euclidean geodesic in $T'_0$ has unit length. So $pr$ is a normalised covering projection with respect to the cusp $c$.
\end{proof}

We show below that the normalised cusp height and the area of the normalised flat cusp torus are both properties of the cusp and do not depend on the choice of the normalised covering projection that is used to define it.

\begin{lemma}\label{normalised_uniqueness}
Let $pr_1: \H^3 \to M$ and $pr_2: \H^3 \to M$ be covering projections that are normalised with respect to the same cusp of $M$. Then there exists an isometry $\phi$ of $\H^3$ which is a parabolic or elliptic isometry that fixes $\infty$ such that $pr_1= pr_2 \circ \phi$. The cusp height with respect to $pr_1$ is equal to the cusp height with respect to $pr_2$ and the area of the flat cusp torus with respect to $pr_1$ is equal to the area of the flat cusp torus with respect to $pr_2$.
\end{lemma}
\begin{proof}
The space $\H^3$ is simply connected and $pr_2: \H^3 \to M$ is a covering map so by the lifting criterion there exists a homeomorphism $\sigma: \H^3 \to \H^3$ such that $pr_2 \circ \sigma = pr_1$. Both $pr_1$ and $pr_2$ are local isometries so $\sigma$ is in fact an isometry. Suppose both $pr_1$ and $pr_2$ take $\infty$ to the cusp $c$ of $M$ but $\sigma(\infty) = A \in \del \H^3$. Then $pr_2(A)=c$ and so there exists a deck transformation $\psi$ of $pr_2$ that takes $A$ to $\infty$. Let $\phi = \psi\circ \sigma$. Then $pr_2 \circ \phi = (pr_2 \circ \psi) \circ \sigma = pr_2 \circ \sigma = pr_1$ and $\phi(\infty) = \infty$.

Suppose $\phi$ is an isometry taking the form $\phi(z) = f_r(z) = r z$ on $\C$, for some $r>0$. If $\delta$ is a deck transformation of $pr_2$, then $\phi^{-1} \circ \delta \circ \phi$ is a deck transformation of $pr_1$. Let $\Gamma_\infty^1=<\alpha_1, \beta_1>$ be the presentation of the cusp group with respect to $pr_1$,  where $\alpha_1 (p) = p + u$ and $\beta_1(p) = p +v$ for independent vectors $u$ and $v$ of $\R^2 \times \{0\}$. Let $P(u, v)$ denote the parallelogram in $\R^2 \times \{0\}$ with sides $u$ and $v$. Let $T_1=P(u, v)/\Gamma_\infty^1$ denote the normalised flat cusp torus with respect to $pr_1$. The group of deck transformations of $pr_2$ which fix $\infty$ is given by $\Gamma_\infty^2 = \phi\, \Gamma_\infty^1\, \phi^{-1}$. The subgroup $\Gamma_\infty^2$ is therefore generated by the parabolics $\alpha_2(p) = p + r u$ and $\beta_2=p + r v$. So the normalised flat cusp torus $T_2 =P(r u, r v)/\Gamma_\infty^2$ has shortest closed geodesic of length $r$ times the length of the shortest closed geodesic of $T_1$ and its area is $r^2$ times the area of $T_1$. Both $pr_1$ and $pr_2$ are normalised so the length of the shortest closed geodesic in $T_1$ and in $T_2$ is $1$. Therefore $r=1$ and $\phi$ is the identity map.

Suppose $\phi$ is an elliptic or parabolic isometry fixing $\infty$. Then $\phi$ preserves horoballs centered at $\infty$, so $\phi(H(h))=H(h)$ for $h>0$. Let $h_1$ and $h_2$ be the normalised cusp heights with respect to $pr_1$ and $pr_2$. Then both  $\phi(H(h_1))/\phi \, \Gamma_\infty^1\, \phi^{-1} = H(h_1)/\Gamma_\infty^2$ and $H(h_2)/\Gamma_\infty^2 $ are isometric to $C$ under $pr_2$. So $h_1=h_2$ as required. Parabolic and elliptic isometries that fix $\infty$ act on $\R^2 \times \{0\}$ as Euclidean isometries, so the area of the flat cusp torus with respect to $pr_1$ is equal to the area of the flat cusp torus with respect to $pr_2$.

Let $g_\theta(z)=e^{i\theta}z$ and let $h_v(z) = z-v$. Suppose $\phi$ is a  loxodromic isometry which fixes the vertical geodesic $(w, \infty)$. Then $\phi$ has the form $\phi(z) = re^{i\theta}(z-w) + w = re^{i\theta}z - v$ where $v=re^{i\theta}w-w$, so $\phi=h_v \circ g_\theta \circ f_r$. Let $\pi: \H^3 \to M$ be the covering projection $\pi=pr_2 \circ h_v \circ g_\theta$. All orientation preserving Euclidean isometries of $\R^2$ are either rotations or translations, so $h_v \circ g_\theta$ is a parabolic or elliptic isometry. By the above arguments $\pi$ is also a covering projection that is normalised with respect to the same cusp as $pr_2$. Also, the cusp height and area of the flat cusp torus with respect to $\pi$ is the same as that with respect to $pr_2$. So replacing $pr_2$ with $\pi$ in the previous argument we can conclude that $r=1$, i.e., $f_r=id$ and $\phi$ is a parabolic or elliptic isometry that fixes $\infty$. 

Every isometry of $\H^3$ is a loxodromic, parabolic or elliptic isometry therefore $\phi$ must be a parabolic or elliptic isometry that fixes $\infty$.
\end{proof}

We will now obtain an upper bound for the normalised cusp height $h_0(c)$ and the area $A_0(c)$ of the normalised flat cusp torus of a cusp $c$.

\begin{lemma}\label{cuspbounds}
Let $M$ be a complete cusped orientable hyperbolic $3$-manifold with volume $V$ and let $c$ be a cusp of $M$. Let $\epsilon$ be a Margulis number for such manifolds. 
Then,
$$h_0(c) \leq 1/\epsilon$$
$$A_0(c) \leq 2V/\epsilon^2$$
\end{lemma}

\begin{proof}
Fix a covering projection $pr: \H^3 \to M$ that is normalised with respect to the cusp $c$. We shall use the notation fixed in Definition \ref{normalised_notation}. Let $w_0=1/\epsilon$ and let $p=(x_0, y_0, z_0) \in H(h_0(c))$. After composing with elements of $\Gamma_\infty$ we may assume that $(x_0, y_0, 0) \in P(u, v)$. Let $\gamma_0(t)=(x(t), y(t), 0)$ be a Euclidean geodesic path (straight line segment) in $P(u,v)$ through $(x_0, y_0, 0)$, such that upon taking quotients $\gamma_0/\Gamma_\infty$ is a shortest closed geodesic in $T_0=P(u,v)/\Gamma_\infty$. Let $\gamma_z(t) = (x(t), y(t), z)$ be a parallel path to $\gamma_0$ at height $z>w_0$. The hyperbolic length of $\gamma_z$ is $1/z$ times the Euclidean length of $\gamma_0$. This is because the hyperbolic metric $\sqrt{(dx)^2 + (dy)^2 + (dz)^2}/z$ reduces to $\sqrt{(dx)^2 + (dy)^2}/z$ on the horosphere $\R^2 \times \{z\}$. The curve $\gamma_0/\Gamma_\infty$ is a shortest closed geodesic in the normalised flat torus $T_0$, so it has unit Euclidean length. Hence the hyperbolic length of $\gamma_z/\Gamma_\infty$ is $1/z<1/w_0=\epsilon$. The essential closed curve $pr(\gamma_z)$  through $pr(p)$ is of hyperbolic length less than $\epsilon$. Consequently, the neighbourhood of radius $\epsilon/2$ about $pr(p)$ is not an embedded ball. The injectivity radius at $pr(p)$ must therefore be less than $\epsilon/2$, i.e., $pr(p)$ lies in the thin part of $M$. The thin part of $M$ is a disjoint union of the cusp neighbourhoods and tubes around short closed geodesics.  The map $pr$ sends $\infty$ to the cusp $c$, so $pr(H(w_0)) \subset C= pr(H(h_0))$. And so $w_0 \geq h_0$ as required.

We next obtain the bound on the Euclidean area $A_0(c)$ of the flat cusp torus $T_0$ of $pr$. The volume of the set $\lift{C}(w_0) = \{(x, y, z) \in \H^3: z>w_0, (x, y) \in P(u,v)\}$ is $\int_{z=w_0}^\infty \int_{(x, y) \in P(u,v)} (1/z^3) dx dy dz = Area(P(u,v))/(2w_0^2)$. The volume of $C(w_0) = H(w_0)/\Gamma_\infty=\lift{C}(w_0)/\Gamma_\infty$ is equal to the volume of $\lift{C}(w_0)$. The cusp neighbourhood $C(w_0)$ is a subset of $C$, so the volume of $\lift{C}(w_0)$ is less than the volume of the manifold $V$. This gives $A_0(c)=Area(P(u,v))\leq 2Vw_0^2=2V/\epsilon^2$ as required.
\end{proof}

\begin{definition}\label{thickdef}
A geometric ideal triangulation $\tau$ of $M$ is a realisation of $M$ as the quotient of a collection of hyperbolic ideal tetrahedra by face pairing isometries such that the tetrahedra of $\tau$ glue together consistently to give the complete hyperbolic structure on $M$. 
We say $\tau$ is $\theta_0$-thick if all the dihedral angles of $\tau$ are at least $\theta_0$. Similarly for a Euclidean triangulation $\bar{\tau}$ of a flat torus, we say $\bar{\tau}$ is $\theta_0$-thick if the angles of all its triangles are at least $\theta_0$. Note that in either case, as $3\theta_0 \leq \pi$, so $0< \theta_0 \leq \pi/3$.
\end{definition}

\begin{remark}\label{sinecomp}
Let $\theta$ be the angle of a $\theta_0$-thick Euclidean triangle. An observation we shall repeatedly use is that $\sin\theta\geq \sin\theta_0$. If $\theta$ lies in $[\theta_0, \pi)$ and $\sin\theta<\sin\theta_0$, then $\theta\in(\pi-\theta_0, \pi)$. This would imply that the sum of angles of the triangle is greater than $(\pi - \theta_0) + 2\theta_0 >\pi$, which is a contradiction. A similar argument holds when $\theta$ is the dihedral angle of a $\theta_0$-thick ideal tetrahedron.
\end{remark}

The following lemma gives bounds on the length of edges of a $\theta_0$-thick triangulation of a flat torus with bounded area and unit length shortest closed geodesic.

\begin{lemma}\label{torusedgebounds}
Let $T$ be a flat torus with area at most $A$ and the shortest closed geodesic of unit length. Let $\bar{\tau}$ be a $\theta_0$-thick Euclidean triangulation of $T$ with  at most $n$ triangles. Then the edge lengths of $\bar{\tau}$ have lower bound $l_0(n)$ and upper bound $L_0(A)$ where
$$l_0(n)=\frac{(\sin\theta_0)^n(\sqrt{n^2+8n}-n)}{4n}$$
$$L_0(A)=2\sqrt{A \cot\theta_0}$$
\end{lemma}
\begin{proof}
Let $L$ be the length of a longest edge of $\bar{\tau}$ and let $[abc]$ be a triangle of $\bar{\tau}$ with the edge $[bc]$ of length $L$. The angles $\angle abc$ and $\angle acb$ are at least $\theta_0$ so a point $p$ can be chosen in $[abc]$ such that $[pbc]$ is an isosceles triangle that lies in $[abc]$ with $\angle pbc=\angle pcb=\theta_0$. The area of the isosceles triangle $[pbc]$ with base $[bc]$ of length $L$ and equal angles $\theta_0$ is $L^2\tan\theta_0/4$. This area is bounded by the area of triangle $[abc]$, which in turn is bounded by $A$. So we get $L \leq 2\sqrt{A\cot\theta_0}$ as required.\\

\begin{figure}
\centering
\def\svgwidth{0.9\columnwidth}
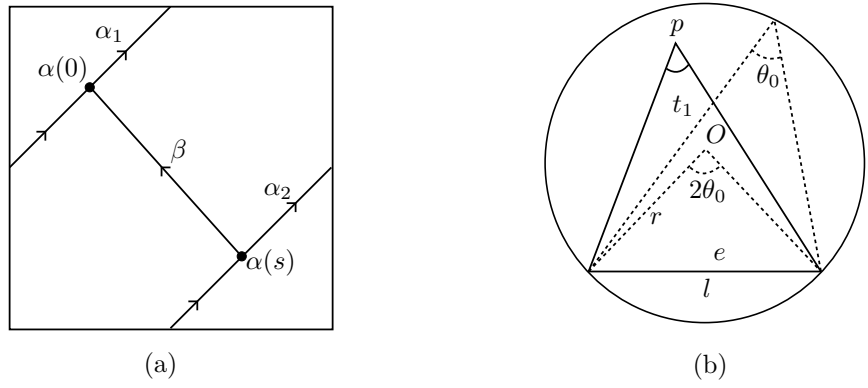
\caption{(a) $\alpha$ is a shortest closed curve in $T$ and $\beta$ is a geodesic from $\alpha(s)$ to $\alpha(0)$. (b) The triangle $t_1$ of $\bar{\tau}$ with a shortest edge $e$ of length $l$. As the angle subtended on the circle is $\theta_0$ and the angle at $p$ is greater than or equal to $\theta_0$ so $p$ lies in this disk of radius $r=l/(2\sin\theta_0)$.}\label{torusedgebounds_fig}
\end{figure}

Let $\alpha:[0,1]\to T$ be a shortest closed geodesic in $T$, which is given to be of unit length. We shall first show that $\alpha$ intersects each triangle of $\bar{\tau}$ in at most  $2L + 1$ components. See the example in Figure \ref{torusedgebounds_fig}(a). For $s\in (0, 1)$, let $\alpha_1=\alpha|_{[0,s]}$ and let $\alpha_2=\alpha|_{[s,1]}$. Let $\beta$ be a geodesic arc in $T$ from $\alpha(s)$ to $\alpha(0)$ which is different from $\bar{\alpha}_1$ and $\alpha_2$, where $\bar{\alpha}_1$ denotes the arc $\alpha_1$ in the reverse direction.  Then $\alpha_1 \star \beta$ and $\bar{\alpha}_2 \star \beta$ are non-trivial closed curves in $T$ and hence both have length at least 1, i.e., $l(\alpha_1) + l(\beta) \geq 1$ and $l(\alpha_2) + l(\beta) \geq 1$. 
These curves are non-trivial as otherwise $\alpha_1 \sim \bar{\beta}$ and $\alpha_2 \sim \beta$ but in the flat metric there exists a unique geodesic up to path-homotopy between any two points. The length of $\alpha$, $l(\alpha)=l(\alpha_1) + l(\alpha_2)=1$ so $l(\alpha_1)\leq 1/2$ or $l(\alpha_2)\leq 1/2$ and therefore $l(\beta) \geq 1/2$. 

Let $\gamma$ be a geodesic that intersects $\alpha$ $k$ times. Let $\beta$ be a segment of $\gamma$ between two consecutive intersection points with $\alpha$. By the above arguments, each such segment of $\gamma$ has length at least 1/2. There are at least $k-1$ such segments, so $(k-1)(1/2)\leq l(\gamma)$. If we assume that $l(\gamma)\leq L$ then we get $k \leq 2L+1$. So any geodesic segment of length at most $L$ intersects $\alpha$ at most $2L + 1$ times.

The geodesic $\alpha$ intersects each triangle $t$ of $\bar{\tau}$ in parallel segments $\delta_i$ inside $t$, so there exists an edge $e$ of $t$ which intersects every $\delta_i$. The length of edge $e$, $l(e) \leq L$ so taking $\gamma$ as $e$ in the above arguments, we can see that $\alpha$ intersects $t$ in at most $k\leq 2L+1$ components. \\

We will now obtain an upper bound for $L$ in terms of the length of the shortest edge.
Let $e$ be a smallest edge of $\bar{\tau}$, with length $l$. Let $t_1$ be a triangle of $\bar{\tau}$ containing edge $e$ and let $p$ be the vertex of $t_1$ opposite to $e$. All angles of $t_1$ are at least $\theta_0$ so $p$ lies in a disk with $e$ as a chord and $\theta_0$ the angle subtended on the boundary circle. Refer to the diagram in Figure \ref{torusedgebounds_fig}(b). A side length of $t_1$ is maximum when it is the diameter of this circle. The radius of this circle is $l/(2\sin\theta_0)$. So sides of $t_1$ have length at most $l/\sin\theta_0$. 

If $t_2$ is a triangle adjacent to $t_1$ then its shortest edge has length $l'\leq l/\sin\theta_0$. By a similar reasoning the lengths of its sides is bounded above by $l'/\sin\theta_0 \leq l/(\sin\theta_0)^2$. The Euclidean triangulation $\bar{\tau}$ has at most $n$ triangles so inductively an upper bound on its edge lengths in terms of $l$ is given by $L \leq l/(\sin\theta_0)^{n}$. To simplify notation, let $c=1/(\sin\theta_0)^n$, so that $L \leq cl$. Note that as $0< \theta_0\leq \pi/3$, so $c>0$.\\

The shortest geodesic $\alpha$ intersects each triangle in at most $2L+1$ components so $\alpha$ is divided into at most $(2L+1)n$ segments by the triangles of $\bar{\tau}$. The length of each such segment is bounded by the diameter of the corresponding triangle which is at most $L$. This gives $1=l(\alpha) \leq (2L+1)nL$. Using the inequality $L \leq cl$, we get the quadratic inequality

$$q(l)=(2nc^2)l^2 + (nc) l - 1 \geq 0$$

The roots of $q(l)$ are $x=(-n- \sqrt{n^2+8n})/(4nc)$ and $y=(-n+ \sqrt{n^2+8n})/(4nc)$. The coeffficient $2nc^2>0$, so $q(l)\geq 0$ for $l\in(-\infty, x]\cup[y, \infty)$ and $q(l)< 0$ for $l\in (x, y)$. As $x<0$ and $l$ is positive so we can conclude that 
$$l\geq y=\frac{(\sin\theta_0)^n(\sqrt{n^2+8n}-n)}{4n}$$

\end{proof}

\begin{lemma}\label{circumrad}
Let $\mathcal{S}$ be the set of all Euclidean triangles with edge length at most $L_0$ and angles at least $\theta_0$. The circumradius of any triangle in $\mathcal{S}$ is at most $L_0/(2\sin\theta_0)$.
\end{lemma}
\begin{proof}
Let $[pqr]$ be a Euclidean triangle with $l([pq])=a$ and $\angle r = \theta$. Then its circumradius is given by $a/(2\sin\theta)$. If $[pqr]\in \mathcal{S}$ then by Remark \ref{sinecomp}, $a/(2\sin\theta) \leq L_0/(2\sin\theta_0)$ as required.
\end{proof}

\begin{definition}
Let $\tau$ be an ideal triangulation of $M$ and let $pr: \H^3 \to M$ be a covering map that is normalised with respect to a cusp of $M$. A lift of $\tau$ with respect to $pr$ is an ideal triangulation $\lift{\tau}$ of $\H^3$ such that $pr$ is a simplicial map from $\lift{\tau}$ to $\tau$. A simplex of $\lift{\tau}$ is called vertical if one of its ideal vertices is $\infty$. We say $\lift{\tau}$ intersects the horosphere $X_z=\R^2 \times \{z\}$ vertically if every simplex of $\lift{\tau}$ that intersects $X_z$ is vertical.

Let $h_0(c)$ be the normalised cusp height of the cusp $c$. Let $z\geq h_0(c)$ be such that $X_z$ intersects $\lift{\tau}$ vertically. We will show in Lemma \ref{taunormalised_existence} that such a $z$ does exist. The infimum $h_0(c, \tau)$ of such $z \geq h_0(c)$ is called the $\tau$-normalised cusp height of $c$. We will show in Lemma \ref{taunormalised_uniqueness} that it does not depend on the choice of the normalised covering projection with respect to $c$.

Note that when $z> h_0(c,\tau)$ then $\bar{\tau} = \lift{\tau} \cap (\R^2 \times \{z\})$ is a Euclidean triangulation of $\R^2 \times \{z\} \subset \R^3$. Also, $\lift{\tau} \cap cl(H(z)) = \bar{\tau} \times [z, \infty)$, i.e., an $n$-simplex $\delta$ of $\lift{\tau} \cap cl(H(z))$ is the union of vertical geodesic rays that begin at points in some fixed $n-1$ simplex $\bar{\delta}$ of $\bar{\tau}$ and asymptotically end at $\infty$.
\end{definition}

We will show below that the $\tau$-normalised cusp height of a cusp $c$ does not depend on the choice of the normalised covering projection map with respect to $c$.
\begin{lemma}\label{taunormalised_uniqueness}
Let $pr_1: \H^3 \to M$ and $pr_2: \H^3 \to M$ be covering projections that are normalised with respect to a cusp $c$ of $M$. Let $\tau$ be an ideal triangulation of $M$. Let $\lift{\tau}_1$ and $\lift{\tau}_2$ be lifts of $\tau$ with respect to $pr_1$ and $pr_2$ respectively. Then the $\tau$-normalised cusp height of $c$ with respect to $\lift{\tau}_1$ is equal to the $\tau$-normalised cusp height of $c$ with respect to $\lift{\tau}_2$.
\end{lemma}
\begin{proof}
By Lemma \ref{normalised_uniqueness}, the maps $pr_1$ and $pr_2$ both send $H(h_0(c))$ to the cusp neighbhourhood $C$ of $c$ consisting of points with injectivity radius less than $\epsilon/2$. And there exists an isometry $\phi$ which is a parabolic or elliptic isometry fixing $\infty$ such that $pr_1 = pr_2 \circ \phi$. The isometry $\phi$ preserves horospheres $X_z = \R^2 \times \{z\}$ centered at $\infty$. So for $z > h_0(c)$, $pr_1(X_z) = pr_2(X_z)$ is a torus $T_z$ in $M$. We will prove that for any $z > h_0(c)$, $\lift{\tau}_1$ intersects $X_z$ vertically  if and only if $\lift{\tau}_2$ intersects $X_z$ vertically.

A geodesic $\gamma$ of $M$ is orthogonal to $T_z$ if and only if every lift of $\gamma$ with respect to $pr_1$ or with respect to $pr_2$ which intersects $X_z$, intersects $X_z$ orthogonally. Assume that $\lift{\tau}_1$ intersects $X_z$ vertically but $\lift{\tau}_2$ does not. Let $\lift{\delta}_2$ be an edge (or face) of $\lift{\tau}_2$ that intersects $X_z$ but is not vertical, i.e., the intersection is not an orthogonal intersection. Then projecting down via $pr_2$ we get an edge (or face) $\delta$ of $\tau$ that intersects $T_z$ non-orthogonally. And hence there exists an edge (or face) $\lift{\delta}_1$ of $\lift{\tau}_1$ that intersects $X_z$ non-orthogonally. This contradicts the assumption that $\lift{\tau}_1$ intersects $X_z$ vertically.
\end{proof}

We will next calculate an upper bound for the $\tau$-normalised cusp height of a cusp $c$ of $M$, where $\tau$ is any $\theta_0$-thick triangulation of $M$. This upper bound $z_0$ is independent of the choice of the $\theta_0$-thick triangulation.

\begin{lemma}\label{taunormalised_existence}
Let $M$ be a complete cusped orientable hyperbolic $3$-manifold with volume $V$. Let $\epsilon$ be a Margulis number for such manifolds. Let $c$ be a cusp of $M$ and let $\tau$ be a $\theta_0$-thick triangulation of $M$. Then,
$$h_0(c, \tau) \leq \frac{\sqrt{2V \cot\theta_0}}{\epsilon\sin\theta_0}$$
\end{lemma}

\begin{proof}
Let $pr:\H^3 \to M$ be a covering projection that is normalised with respect to the cusp $c$. Let $\lift{\tau}$ be a lift of $\tau$ to $\H^3$ with respect to $pr$. Let $abc$ denote a non-vertical face of $\lift{\tau}$ with vertices $a$, $b$, $c$ in $\R^2 \times \{0\}$. The face $abc$ lies on a hemisphere with center on the plane $\R^2\times \{0\}$ and radius equal to the circumradius of the Euclidean triangle $[abc]$ in $\R^2 \times \{0\}$. We call this circumradius the circumradius of the face $abc$. The maximum Euclidean height reached by $abc$ is bounded by the circumradius of $abc$.

We first show that there exists some $z>h_0(c)$ such that $X_z = \R^2 \times \{z\}$ intersects $\lift{\tau}$ vertically. Let $\Gamma_\infty$ denote the cusp group of $c$. Take $z>h_0(c)$ and let $T_z=pr(X_z)=X_z/\Gamma_\infty$ be a cusp torus in the thin cusp neighbourhood $C$ of $c$. As $\tau$ is a finite triangulation and each simplex of $\tau$ intersects $T_z$ finitely many times, so  $X_z$ is intersected by finitely many simplexes $\delta_i$ of $\lift{\tau}$ along with all their translates by $\Gamma_\infty$. The cusp group $\Gamma_\infty$ is generated by parabolic isometries (which induce Euclidean isometries on $\R^2 \times \{0\}$), so the circumradii of non-vertical faces of $\delta_i$ are unchanged under the action of $\Gamma_\infty$. By taking $z>h_0(c)$ larger than the circumradii of these finitely many non-vertical faces of $\delta_i$, we may assume that $\lift{\tau}$ intersects $X_z$ vertically. In particular, $X_z\cap \lift{\tau}$ induces a $\theta_0$-thick Euclidean triangulation of $X_z$. Orthogonally projecting it down to $\R^2 \times \{0\}$ gives a $\theta_0$-thick Euclidean triangulation $\bar{\tau}$ of $\R^2 \times \{0\}$. The lifted triangulation $\lift{\tau}$ is invariant under deck transformations so $\bar{\tau}/\Gamma_\infty$ gives a Euclidean triangulation of the normalised flat cusp torus $T_0=\R^2 \times \{0\}/\Gamma_\infty$. 

Let $L_0$ be the length of the longest edge of $\bar{\tau}$ and let $A_0(c)$ denote the Euclidean area of $T_0$. By combining the bounds of Lemma \ref{circumrad}, Lemma \ref{torusedgebounds} and Lemma \ref{cuspbounds}, the circumradius of a triangle of $\bar{\tau}$ is less than
$$\frac{L_0}{2\sin\theta_0} \leq \frac{2 \sqrt{A_0(c) \cot\theta_0}}{2\sin\theta_0} \leq \frac{\sqrt{2V\cot\theta_0}}{\epsilon\sin\theta_0}$$

Let $z_0 = \sqrt{2V\cot\theta_0}/(\epsilon\sin\theta_0)$. Then for any $z\geq z_0$, $X_z$ intersects $\lift{\tau}$ vertically. Note that $g(\theta_0)=\sqrt{\cot\theta_0}/\sin\theta_0$ is a decreasing function of $\theta_0$ in $(0, \pi/2)$. The angle $\theta_0\in (0,\pi/3]$ so the minimum value for $g(\theta_0)$ is $g(\pi/3)>0.8$. The minimum volume of a cusped hyperbolic manifold is the volume of the Figure-Eight knot complement, which is $2v_{tet}$ where $v_{tet} >1$ is the volume of the regular ideal hyperbolic tetrahedron\cite{CaoMey}. Therefore $z_0 > 2(0.8)/\epsilon> 1/\epsilon$ and by Lemma \ref{cuspbounds}, $1/\epsilon \geq h_0(c)$. And so $h_0(c, \tau) \leq z_0$.
\end{proof}

\begin{definition}
Let $pr:\H^3 \to M$ be a normalised covering projection with respect to a cusp $c$ of $M$. Let $\tau$ be a $\theta_0$-thick triangulation of $M$. For any $z>h_0(c, \tau)$, $T_z = pr(\R^2 \times \{z\})$ is an embedded torus in $M$ (as $h_0(c, \tau) \geq h_0(c)$). Note that the height $z$ does not depend on the choice of the normalised covering projection with respect to $c$.
The hyperbolic metric of $M$ induces a Euclidean metric on $T_z$ and $\bar{\tau}_z=\tau \cap T_z$ is a Euclidean triangulation of $T_z$. If $C(z)$ denotes the cusp neighbourhood of $M$ bounded by $T_z$ and $cl(C(z))$ its closure, then $\tau \cap cl(C(z))$ is the triangulation given by coning $\bar{\tau}$ with the cusp point $c$. In other words, an $n$-simplex $\delta$ of $\tau \cap cl(C(z))$ is the union of geodesic rays perpendicular to $T_z$ that begin at points in some fixed $n-1$ simplex $\bar{\delta}$ of $\bar{\tau}$ and asymptotically end at the cusp.
\end{definition}

\begin{remark}\label{constants}
Let $\tau$ be a geometric ideal $\theta_0$-thick triangulation of a cusped hyperbolic manifold $M$ with $k$ cusps and $m$ tetrahedra. Let $z \geq z_0$ and let $X_z = \R^2 \times \{z\}$. Let $pr: \H^3 \to M$ be a covering projection which is normalised with respect to a cusp $c$ of $M$. Then $T_z=pr(X_z)$ is a cusp torus in the thin cusp neighbourhood $C$ of $c$ in $M$. The induced metric on $T_z$ is Euclidean and the triangulation $\tau$ induces a Euclidean triangulation $\bar{\tau}$ on $T_z$. Each tetrahedron of $\tau$ intersects $T_z$ in at most $4$ triangles so the number of triangles in $\bar{\tau}$ is at most $4m$. Let $\lift{\tau}$ be a lift of $\tau$ with respect to $pr$. The triangulation $\lift{\tau} \cap X_z$ is invariant under $\Gamma_\infty$, so orthogonally projecting to $\R^2 \times \{0\}$ we can identify $\bar{\tau}$ with a Euclidean triangulation of the flat torus $T_0 = \R^2 \times \{0\}/\Gamma_\infty$. Taking $n= 4m$ in Lemma \ref{torusedgebounds} we get the lower bound on the Euclidean lengths of edges of $\bar{\tau}$ in $T_0$ as 
$$l_0 = \frac{(\sin \theta)^{4m}(\sqrt{16m^2+32m}-4m)}{16m} = \frac{(\sin \theta)^{4m}(\sqrt{m^2+2m}-m)}{4m}$$

The volume of any ideal hyperbolic tetrahedron is at most $v_{tet}$, the volume  of  the regular ideal tetrahedron. So putting $V \leq mv_{tet}$ in Lemma \ref{taunormalised_existence} we can take the upper bound on $h_0(c, \tau)$ to be

$$z_0=\frac{\sqrt{2m v_{tet} \cot\theta_0}}{\epsilon\sin\theta_0}$$
 
And finally putting $V \leq mv_{tet}$ in Lemma \ref{cuspbounds} we get a bound on the area of the flat cusp torus to be
$$A_0 = 2mv_{tet}/\epsilon^2$$

For any geometric ideal $\theta_0$-thick triangulation with $m$ tetrahedra, we fix these as the values for $l_0$, $z_0$ and $A_0$ for the rest of this paper. Note that all these constants are independent of the choice of the cusp and the normalised covering map. 

We also fix the following notations: Let $pr_i$ be a covering projection that is normalised with respect to the $i$-th cusp. Let $P_i(u, v)$ be the parallelogram in $\R^2 \times \{0\}$ spanned by the vectors which generate the normalised presentation of the cusp group $\Gamma_\infty$. Let $\lift{C}_i(z_0) = \{(x, y, z)\in \H^3: (x, y) \in P_i(u, v), z> z_0\}$.  Let $C_i(z_0) = pr_i(\lift{C}_i(z_0))=pr_i(H(z_0))$ and let $C(z_0) = \cup_{i=1}^k C_i (z_0)$. Then $M(z_0) = M \setminus C(z_0)$ is a compact hyperbolic manifold with flat tori boundary such that $\tau$ induces a Euclidean triangulation on $\del M(z_0)$.
\end{remark}

\subsection{Lower bound on systole length}
Let $\tau$ be a $\theta_0$-thick triangulation of $M$. The aim of this subsection is to  obtain a lower bound on the systole length of $M$ (Theorem \ref{mainthm3}).

The following simple calculation will be used repeatedly:
\begin{lemma}\label{distfromline}
Let $p=(x, y)\in \H^2$ and let $Y=\{(0,y): y>0\}$. Then the distance between $p$ and $Y$ in $\H^2$ is given by  $d(p, Y) = \arcsinh(x/y)$.

\end{lemma}
\begin{proof}
The shortest geodesic from $p$ to $Y$ is the arc of a circle through $p$ perpendicular to both the $x$-axis and $Y$, so it is the geodesic segment $\gamma$ joining $p$ and $(0,r)$ where $r=\sqrt{x^2 + y^2}$. The distance between these points is given by 

$$
d(p, Y)= \int_{\gamma}\frac{\sqrt{1+(dy/dx)^2}}{y} dx = \int_{0}^{x} \frac{r}{r^2-x^2} dx = \arctanh\left(\frac{x}{r}\right)
= \arcsinh\left(\frac{x}{y}\right)
$$

\end{proof}

We give here some definitions of standard terms in combinatorial topology which we extend to geometric triangulations of hyperbolic manifolds.

\begin{definition}\label{stardef}
Let $\lift{\tau}$ be a geometric triangulation of $\H^3$ and let $\lift{A}$ and $\lift{B}$ be disjoint simplexes of $\lift{\tau}$. We define their join $\lift{A} \star \lift{B}$ as the simplex obtained by taking the union of all geodesics joining points in $\lift{A}$ with points in $\lift{B}$. We define the link of $\lift{A}$ in $\lift{\tau}$ as the union of ideal simplexes $\lift{B} \in \lift{\tau}$ such that $\lift{A} \cap \lift{B}=\phi$ in $\H^3 \cup \del \H^3$ (i.e., they are disjoint in $\H^3$ and do not not have a common ideal vertex in $\del \H^3$) and  $\lift{A} \star \lift{B}$ is an ideal 3-simplex in $\lift{\tau}$. We denote the link of $\lift{A}$ in $\lift{\tau}$ as $lk(\lift{A}, \lift{\tau})$. The (closed) star of $\lift{A}$ in $\lift{\tau}$ is defined by $star(\lift{A}, \lift{\tau})=\lift{A} \star lk (\lift{A}, \lift{\tau})$. The open star of $\lift{A}$ in $\lift{\tau}$ is the interior of $star(\lift{A}, \lift{\tau})$ and is denoted by $instar(\lift{A}, \lift{\tau})$.

Let $pr: \H^3 \to M$ be a covering map. Let $\tau$ be a triangulation of $M$ and let $\lift{\tau}$ be an ideal triangulation of $\H^3$ that is sent to $\tau$ by $pr$. The link, star and open star of a simplex $A$ in $\tau$ is defined respectively by $lk (A, \tau) = pr(lk(\lift{A}, \lift{\tau}))$, $star(A, \tau)=pr(star(\lift{A}, \lift{\tau}))$ and $instar(A, \tau)=pr(instar(\lift{A}, \lift{\tau})))$, where $\lift{A}$ is some lift of $A$ to $\lift{\tau}$. When the triangulation $\tau$ is unambiguous we drop it from the notation and just refer to links, stars and open stars of $A$ as $lk(A)$, $star(A)$ and $instar(A)$. The metric on the links, stars and open stars of $A$ is the subspace metric induced from the hyperbolic metric on $M$. We call a geometric ideal triangulation $\tau$ of $M$ simplicial if for every simplex $A$ of $\tau$, $pr:star(\lift{A}, \lift{\tau}) \to star(A, \tau)$ is an isometry.
\end{definition}

Ideal geometric triangulations may not in general be simplicial. For example, the Gieseking manifold $M_G$ is a cusped non-orientable complete hyperbolic $3$-manifold obtained by identifying the faces of a regular ideal tetrahedron in pairs. It therefore has an ideal triangulation $\tau$ consisting of one tetrahedron, two faces and one edge $E$. The link of $\lift{E}$ in $\lift{\tau}$ is a circle made up of edges all of which are lifts of $E$. So $lk(E)=E$ and $star(E)=M_G$.\\

In the next two lemmas we calculate lower bounds on the injectivity radius of points in $M(z_0)=M \setminus \cup_i C_i(z_0)$ that lie in the $2$-skeleton of $\tau$. Recall that $C_i(z_0)=pr_i(H(z_0))$ in $M$, where $pr_i$ is a covering projection that is normalised with respect to the $i$-th cusp. We will use the notation $d$ for the hyperbolic distance in $\H^3$ and in $M$. For $p \in M$ (or $p \in \H^3$) we will use the notation $N(p, r)$ for the set of points $q$ in $M$ (or $\H^3$) such that $d(p, q) <r$. 

\begin{lemma}\label{edgeball}
Let $\tau$ be a geometric ideal $\theta_0$-thick triangulation of $M$ with $m$ tetrahedra. Let $l_0$ and $z_0$ be as in Remark \ref{constants} and let  $a_0=\arcsinh(l_0\sin\theta_0/z_0)$. Let $E$ be an edge of $\tau$ and let $p$ be a point of $E\cap M(z_0)$. Let $\lift{\tau}$ be a lift of the triangulation $\tau$ to $\H^3$ with respect to some covering projection. If $\lift{E}$ is a lift of $E$ in $\lift{\tau}$ and $\lift{p}$ is the lift of $p$ in $\lift{E}$, then $N(\lift{p}, a_0)$ is an embedded ball in $instar(\lift{E})$ and $N(p, a_0/2)$ is an embedded ball in $M$.
\end{lemma}

\begin{proof}
\begin{figure}
\centering
\def\svgwidth{0.9\columnwidth}
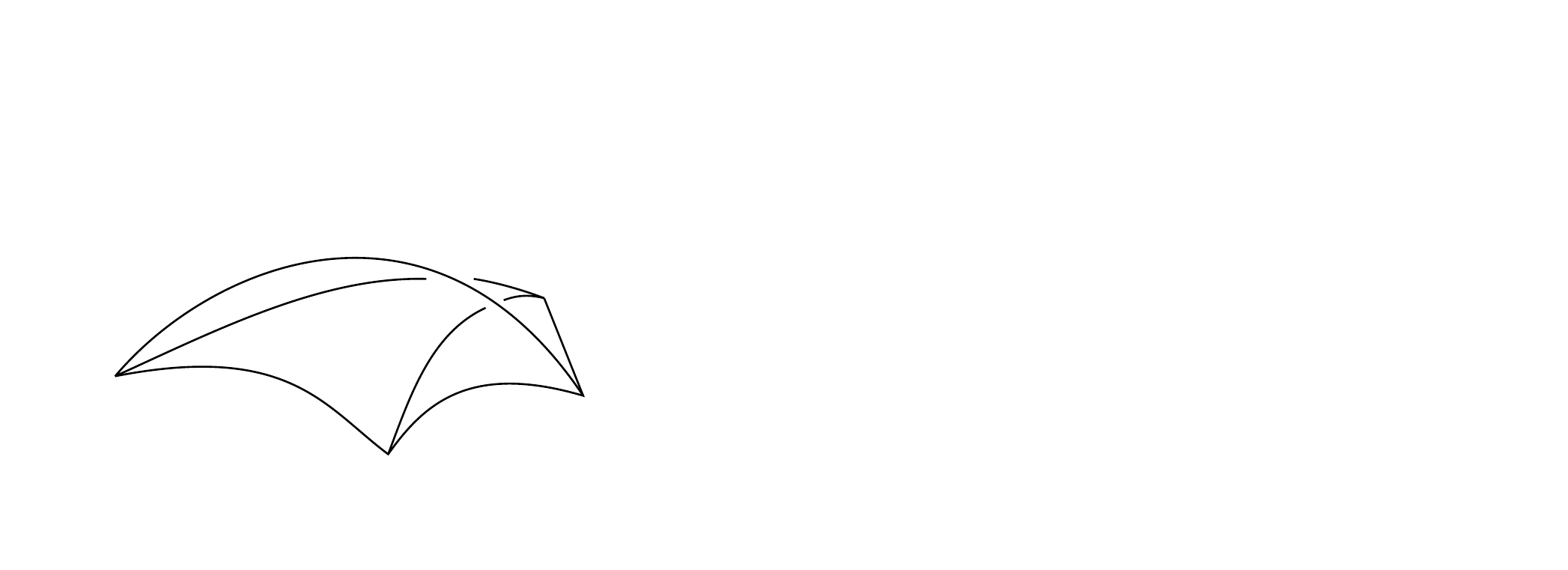
\caption{(a) $\lift{E}=[uv]$ is a lift of $E$ in $\lift{\tau}$ and $\lift{s}\in \lift{F}$ is the closest point to $\lift{p}\in \lift{E}$ in $\del star(\lift{E})$. (b) $\lift{E}'=[v'\infty]$ is a lift of $E$ in $\lift{\tau}'$ and $\lift{s}' \in \lift{F}'$ is the closest point to $\lift{p}'$ in $\del star(\lift{E}')$. }\label{edgeball_fig}
\end{figure}

Let $\pi: \H^3 \to M$ be a covering map and let $\lift{\tau}$ be a lift of $\tau$ to $\H^3$ with respect to $\pi$. Let $\lift{E}$ and $\lift{p}\in \lift{E}$ be the lifts of $E$ and $p$ in $\lift{\tau}$. See Figure \ref{edgeball_fig}(a) for a diagram. The star of $\lift{E}$ in $\lift{\tau}$ is a closed ball. Let $\lift{s}$ be a closest point of $\del star(\lift{E})$ to $\lift{p}$. Let $\lift{F}$ be the ideal 2-simplex of $\del star(\lift{E})$ which contains $\lift{s}$. Let $\lift{E}=[uv]$ (for $u, v \in \del \H^3$) then either $u$ or $v$ is an ideal vertex of $\lift{F}$. Assume that $u$ is an ideal vertex of $\lift{F}$ and let $\lift{\Delta}$ be the simplex of $\lift{\tau}$ that contains $\lift{E}$ and $\lift{F}$. 
If $\pi(u)$ is the cusp $c$ of $M$, then choose a covering projection $pr: \H^3 \to M$ which is normalised with respect to $c$. Let $\lift{\tau}'$ be the lift of $\tau$ with respect to $pr$. Let $\lift{E}'$ be a lift of $E$ with one endpoint at $\infty$. Let $\lift{F}'$, $\lift{p}'$ and $\lift{s}'$ be the corresponding lifts in the simplex $\lift{\Delta}'$ of $\lift{\tau}'$. Therefore both $\lift{E}'$ and $\lift{F}'$ are vertical simplexes of $\lift{\tau}'$  in $\H^3$. See Figure \ref{edgeball_fig}(b) for a diagram. 

For $z\geq z_0\geq h(c, \tau')$, let $X_z = \R^2 \times \{z\}$. By definition of $h(c, \tau')$, the induced triangulation $\lift{\tau}' \cap X_z$ of $X_z$ is a Euclidean triangulation that is invariant under the action of the cusp group $\Gamma_\infty$. Let $\sigma: \R^3 \to \R^2 \times \{0\}$ be the orthogonal projection $(x, y, z) \to (x, y, 0)$. Let $\bar{\tau} = \sigma(\lift{\tau}'\cap X_z)/\Gamma_\infty$  be the induced Euclidean triangulation of the the normalised flat cusp torus $T_0 = \R^2 \times \{0\}/\Gamma_\infty$. By Lemma \ref{torusedgebounds}, the Euclidean length of edges of $\bar{\tau}$ is bounded below by $l_0$.

Let $\bar{p}=\sigma(\lift{E}' \cap X_z)/\Gamma_\infty$ be a vertex of $\bar{\tau}$ and let $[\bar{q}\bar{r}]= \sigma(\lift{F}' \cap X_z)/\Gamma_\infty$ be an edge of $\bar{\tau}$. The lift $\lift{\Delta}'$ of $\Delta$ under $pr$ is the 3-simplex of $\lift{\tau}'$ containing $\lift{E}'$ and $\lift{F}'$ so $\lift{\Delta}' \cap X_z$ orthogonally projects down in $\R^2 \times \{0\}$ to the triangle $[\bar{p}\bar{q}\bar{r}]$ of $\bar{\tau}$. Let $\bar{d}$ denote the Euclidean distance in the flat torus $T_0$. The Euclidean distance $\bar{d}(\bar{p}, [\bar{q}\bar{r}]) \geq l_0\sin\theta_0$ as the angle at $\bar{q}$ is at least $\theta_0$ and $l([\bar{p}\bar{q}]) \geq l_0$.  Let $z(\lift{p}')$ denote the $z$-component of $\lift{p}'$. The point $p$ lies in $M(z_0)$ so $z(\lift{p}') \leq z_0$. Let $H$ be the vertical geodesic plane containing $\lift{p}'$ and $\lift{s}'$. Let $\bar{s}=\sigma(\lift{s}') \in [\bar{q}\bar{r}]$. Let $Y \supset H \cap \lift{F}'$ be the vertical geodesic in $H$ through $\lift{s}'$. The plane $H$ is isometric to $\H^2$ so by Lemma \ref{distfromline} we get $d(\lift{p}', \del star(\lift{E}'))=d(\lift{p}', \lift{s}')= d(\lift{p}', Y) \geq \arcsinh(x/ z(\lift{p}'))$ where $x=\bar{d}(\bar{p}, \bar{s})$ is the horizontal distance between $\lift{p}'$ and $Y$. The ratio $x/z(\lift{p}') \geq \bar{d}(\bar{p}, [\bar{q}\bar{r}])/z_0 \geq l_0\sin\theta_0/z_0$. The function $\arcsinh$ is increasing so $d(\lift{p}', \del star(\lift{E}')) \geq \arcsinh(l_0\sin\theta_0/ z_0)=a_0$. So $N(\lift{p}', a_0)$ is an embedded ball in $instar(\lift{E}')$. In particular the distance between $\lift{p}'$ and any edge of $\lift{\tau}$ other than $\lift{E}'$ is greater than or equal to $a_0$.

Recall that $\lift{\tau}$ is the lift of $\tau$ with respect to the give covering map $\pi$.
Let $instar(\lift{E})$ denote the interior of the star of $\lift{E}$ in $\lift{\tau}$. We will now argue that $N(\lift{p}, a_0)$ is an embedded ball in $instar(\lift{E})$ as well. There exists an isometry $f$ of  $\H^3$ such that $pr =\pi \circ f$. Both the triangulations $f(\lift{\tau}')$ and $\lift{\tau}$ are lifts of $\tau$ under $\pi$ so they are related by a deck transformation $g$ of $\pi$, i.e., $g\circ f (\lift{\tau}')=\lift{\tau}$. Let $\lift{E}' = f^{-1} \circ g^{-1} (\lift{E})$. Then $\lift{E}'$ is a lift of $E$ under $pr$ to an edge of $\lift{\tau}'$. Let $\lift{p}' = f^{-1} \circ g^{-1}(\lift{p})$. Then $\lift{p}'$ is the point of $\lift{E}'$ that is a lift of the point $p$ under $pr$. By above arguments, $N(\lift{p}', a_0)$ is an embedded ball in $\instar(\lift{E}')$. So $g \circ f(N(\lift{p}', a_0)) = N(\lift{p}, a_0)$ is an embedded ball in $\instar(\lift{E})=g\circ f \instar(\lift{E}')$, as $g \circ f (\lift{\tau}')=\lift{\tau}$ and $g \circ f$ is an isometry of $\H^3$.
\newline

\begin{figure}
\centering
\def\svgwidth{0.9\columnwidth}
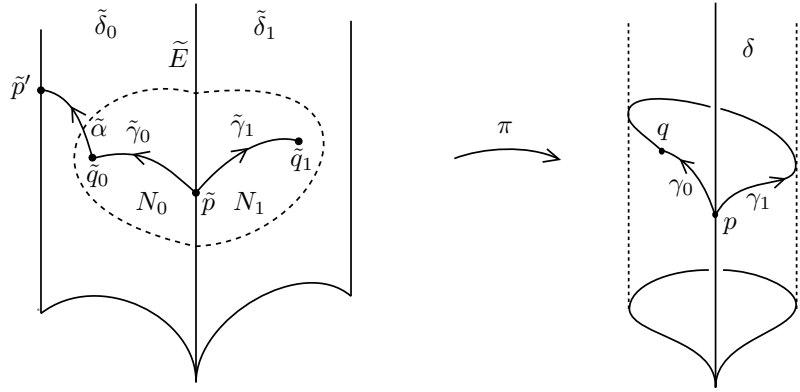
\caption{The points $\lift{q}_i \in \lift{\delta}_i$ are both lifts of $q \in \delta$. The points $\lift{p}'$ and $\lift{p}$ are lifts of $p$. The path $\lift{\alpha}$ is a lift of $\gamma_0 \star \bar{\gamma}_1$ of length less than $a_0$.}\label{edgeball_fig2}
\end{figure}

To prove that $N(p, a_0/2)$ is an embedded ball in $M$ we shall show that the covering projection $\pi:\H^3 \to M$ restricts to an injection on $N(\lift{p}, a_0/2)$ which is a ball in $\H^3$. 
Suppose for points $\lift{q}_0$ and $\lift{q}_1$ in $N(\lift{p}, a_0/2)$,  $\pi(\lift{q}_0)=\pi(\lift{q}_1)=q$. Assume that $q$ lies in the relative interior of the simplex $\delta$ of $\tau$. The map $\pi$ restricted to the relative interiors of simplexes of $\lift{\tau}$ is injective so $\lift{q}_0$ and $\lift{q}_1$ lie in the relative interior of distinct simplexes $\lift{\delta}_0$ and $\lift{\delta}_1$ of $star(\lift{E})$, both of which are lifts of $\delta$. Refer to Figure \ref{edgeball_fig2} for a diagram. Let $N_i=\{\lift{p}\} \cup (int(\lift{\delta}_i) \cap N(\lift{p}, a_0/2))$, for $i=0,1$. As $N_i$ is convex so let $\lift{\gamma}_i$ be a geodesic in $N_i$ from $\lift{p}$ to $\lift{q}_i$ which is of length less than $a_0/2$. The interiors of $\lift{\delta}_i$ are disjoint so $N_0 \cap N_1=\{\lift{p}\}$ and $\lift{\gamma}_0$ intersects $\lift{\gamma}_1$ only at $\lift{p}$. 
The map $\pi$ restricted to a small enough neighbourhood of $\lift{p}$ is an isometry, so $\gamma_0=\pi(\lift{\gamma}_0)$ and $\gamma_1=\pi(\lift{\gamma}_1)$ are different geodesics in $\delta$ from $p$ to $q$. Distinct geodesics between $p$ and $q$ can not be homotopic so $\alpha=\gamma_0 \star \bar{\gamma}_1$ is a non-trivial curve in $\delta$ through $p$ of length less than $a_0$. Lifting $\alpha$ to $\H^3$ now, we get a path $\lift{\alpha}$ in $\lift{\delta}_0$ of length less than $a_0$ from $\lift{p}$ to another lift of $p$. Hence the distance from $\lift{p}$ to an edge of $\lift{\delta}_0$ other than $\lift{E}$ is less than $a_0$, which is a contradiction. Therefore $\pi$ restricted to $N(\lift{p}, a_0/2)$ is injective and hence $N(p, a_0/2)$ is an open ball.
\end{proof}

\begin{lemma}\label{faceball}
Let $\tau$ be a geometric ideal $\theta_0$-thick triangulation of $M$ with $m$ tetrahedra. Let $z_0$ be as in Remark \ref{constants} and for $t>0$ let $r(t)= \arcsinh(\sinh(t)\sin\theta_0)$. Let $F$ be a face of $\tau$ and let $p\in F\cap M(z_0)$ be such that $d(p, \del F) \geq t$. Let $\lift{\tau}$ be a lift of $\tau$ to $\H^3$ with respect to some covering projection. Let $\lift{F}$ be a lift of $F$ in $\lift{\tau}$ and let $\lift{p}$ be the lift of $p$ in $\lift{F}$. Then $N(\lift{p}, r(t))$ is an embedded ball in $instar(\lift{F})$ and $N(p, r(t)/2)$ is an embedded ball in $M$.
\end{lemma}

\begin{proof}
This proof is similar to that of Lemma \ref{edgeball}. For ease of notation we denote $r(t)$ simply by $r$. We shall first show that the distance between $\lift{p}$ and any other face of $\lift{\tau}$ is greater than or equal to $r$ and then argue that the projection map restricted to $N(\lift{p}, r/2)$ is an injection.

\begin{figure}
\centering
\def\svgwidth{0.9\columnwidth}
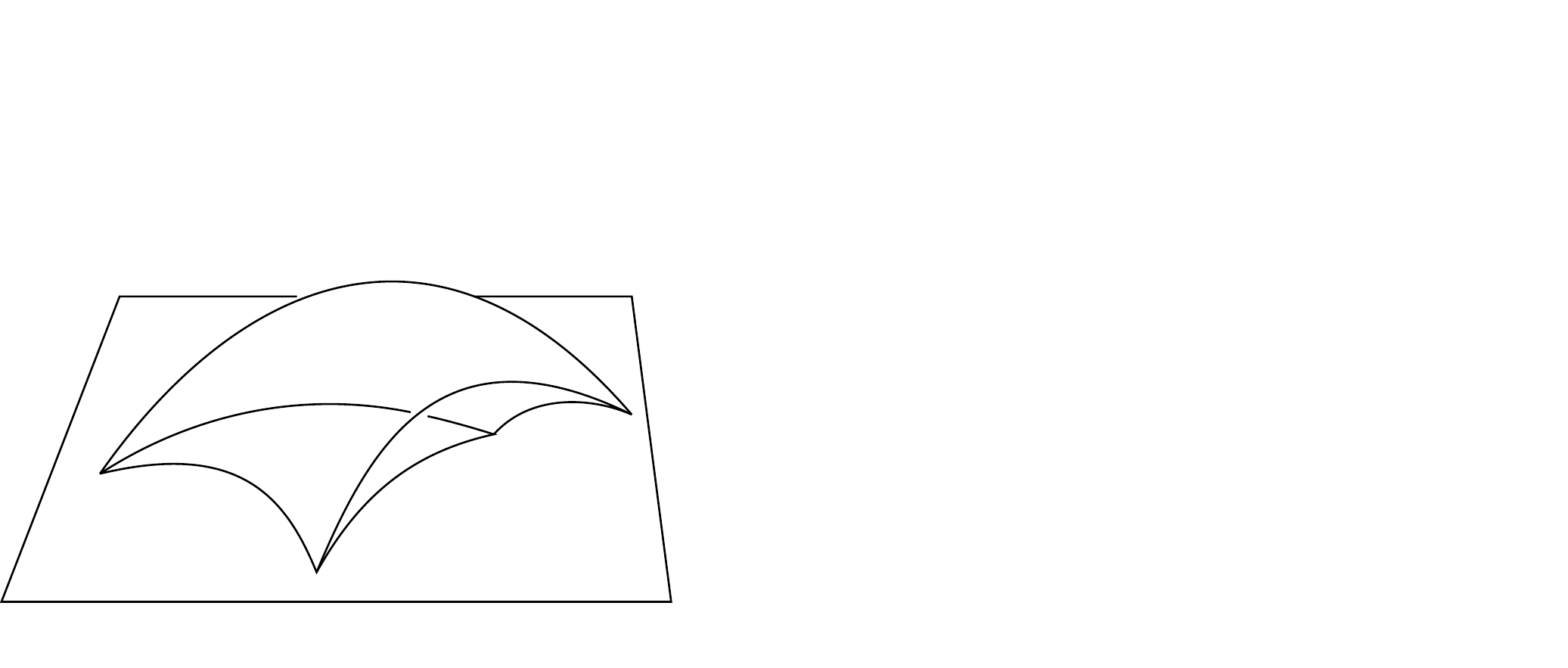
\caption{(a) $\lift{F}$ is a lift of $F$ in $\lift{\tau}$ and $\lift{s}\in \lift{G}$ is the closest point to $\lift{p}\in \lift{F}$ in $\del star(\lift{F})$. (b) $\lift{E}'$ is a lift of $E$ in $\lift{\tau}'$ and is the $z$-axis, $\lift{F}'$ is the lift of $F$ in $\lift{\tau}'$ lying in the $xz$-plane so that $\lift{p}'$ has coordinates $(x, 0, z)$.}\label{faceball_fig}
\end{figure}

Let $\pi: \H^3 \to M$ be a covering map and let $\lift{\tau}$ be a lift of $\tau$ to $\H^3$ with respect to $\pi$. The star of $\lift{F}$ in $\lift{\tau}$ is a closed ball. Let $\lift{s}$ be a closest point to $\lift{p}$ on $\del star(\lift{F})$. Suppose that $\lift{s}$ lies in a face $\lift{G}$ of $\del star(\lift{F})$. Assume that $\lift{G}$ intersects $\lift{F}$ in the edge $\lift{E}=[uv]$ (for $u$, $v$ in $\del \H^3$). Let $\lift{\Delta}$ be the 3-simplex of $\lift{\tau}$ that contains $\lift{F}$ and $\lift{G}$. See Figure \ref{faceball_fig}(a) for a diagram.

If $\pi(u)$ is the cusp $c$ of $M$, then choose a covering projection $pr: \H^3 \to M$ which is normalised with respect to $c$. Let $\lift{\tau}'$ be the lift of $\tau$ with respect to $pr$. Let $\lift{F}'$ and $\lift{G}'$ be lifts of $F$ and $G$ with a common edge $\lift{E}'$ such that $\infty$ is an endpoint of $\lift{E}'$. Both $\lift{F}'$ and $\lift{G}'$ are now vertical ideal triangles in $\H^3$, as in Figure \ref{faceball_fig}(b). We may also assume after composing with a suitable parabolic isometry which fixes $\infty$ and an elliptic isometry with axis the vertical line from $0$ to $\infty$ that $\lift{E}'$ lies along the $z$-axis and $\lift{F}'$ lies in the $xz$-plane. Note that these isometries do not affect the value of $z_0$ and the Euclidean length of the shortest closed geodesic on the normalised flat cusp torus $T_0$ of $pr$ is still $1$.

Let $\lift{p}'\in \lift{F}'$ have the coordinates $(x,0,z)$, then taking the $z$-axis as $Y$ in Lemma \ref{distfromline}, $\arcsinh(x/z)=d(\lift{p}', Y)=d(\lift{p}', \lift{E}')\geq d(p, E) \geq t$. Taking $\sinh$ on both sides of this inequality we get $x \geq z \sinh(t)$. Let $X_z = \R^2 \times \{z\}$, let $e_G=\lift{G}' \cap X_z$ and $e_F=\lift{F}' \cap X_z$. The dihedral angle between $\lift{F}'$ and $\lift{G}'$ at $\lift{E}'$ is at least $\theta_0$. Let $\bar{d}$ denote the Euclidean distance in $X_z$ induced from $\R^3$ (not from $\H^3$). Then $\bar{d}(\lift{p}', e_G) \geq x\sin\theta_0 \geq z\sinh(t)\sin\theta_0$. By Lemma \ref{distfromline} again, $d(\lift{p}', \lift{G}')= \arcsinh(\bar{d}(\lift{p}', e_G)/z) \geq \arcsinh(z\sinh(t)\sin\theta_0/ z)=r$. Therefore as $d(\lift{p}', \del star(\lift{F}')) =d(\lift{p}', \lift{G}')\geq r$ so $N(\lift{p}', r)$ is an embedded ball in $instar(\lift{F}')$. In particular the distance between $\lift{p}'$ and any face of $\lift{\tau}'$ other than $\lift{F}'$ is greater than or equal to $r$. By the same arguments as in Lemma \ref{edgeball}, as there is an isometry of $\H^3$ taking $\tau'$ to $\tau$ so $N(\lift{p}, r)$ is also an embedded ball in $instar(\lift{F})$ in $\lift{\tau}$, which is the lift of $\tau$ with respect to the given covering map $\pi: \H^3 \to M$. 

We shall next show that the covering projection $\pi:\H^3 \to M$ restricts to an injection on $N(\lift{p}, r/2)$. Suppose for points $\lift{q}_0$ and $\lift{q}_1$ in $N(\lift{p}, r/2)$,  $\pi(\lift{q}_0)=\pi(\lift{q}_1)=q$. Assume that $q$ lies in the relative interior of the simplex $\delta$ of $\tau$. The map $\pi$ restricted to the relative interiors of simplexes is injective so $q_0$ and $q_1$ lie in the interior of distinct simplexes $\lift{\delta}_0$ and $\lift{\delta}_1$ of $star(\lift{F})$, both of which are lifts of $\delta$. Let $N_i=\{\lift{p}\} \cup int(\delta_i) \cap N(\lift{p}, r/2)$, for $i=0,1$. $N_i$ is convex so let $\lift{\gamma}_i$ be a geodesic in $N_i$ from $\lift{p}$ to $\lift{q}_i$ which is of length less than $r/2$. $N_0 \cap N_1=\{\lift{p}\}$ so $\lift{\gamma}_0$ and $\lift{\gamma}_1$ intersect only at $\lift{p}$.

The map $\pi$ restricted to a small enough neighbourhood of $\lift{p}$ is an isometry, so $\gamma_0=\pi(\lift{\gamma}_0)$ and $\gamma_1=\pi(\lift{\gamma}_1)$ are different geodesics in $\delta$ from $p$ to $q$. No two distinct geodesics between $p$ and $q$ can be homotopic in $M$, so $\alpha=\gamma_0 \star \bar{\gamma}_1$ is a non-trivial curve in $\delta$ through $p$ of length less than $r$. Lifting $\alpha$ to $\H^3$, we get a path $\lift{\alpha}$ in $\lift{\delta}_0$ of length less than $r$ from $\lift{p}$ to a face of $\lift{\delta}_0$ other than $\lift{F}$. But as the distance between $\lift{p}$ and any other face of $\lift{\tau}$ is greater than or equal to $r$ so we get a contradiction. Therefore $\pi$ restricted to $N(\lift{p}, r/2)$ is injective and hence $N(p, r/2)$ is an embedded ball in $M$.
\end{proof}

Let $Y$ be a vertical geodesic in $\H^2$ and let $p$ be a point of $\H^2$ with $d(p, Y)<t$. If $q$ is the point of $Y$ closest to $p$ then it is clear that $N(p, t) \subset N(q, 2t)$. In the below lemma we show that for small enough $t$, when $q$ is a point of $Y$ at the same height as $p$ then $N(p, r(t)) \subset N(q, 2t)$.
\begin{lemma}\label{sameht}
Let $Y=\{(0,y) \in \H^2: y>0\}$ and $0<t<1/2$. Let $p=(x, y_0)$ and let $q=(0,y_0)$. If $d(p, Y)<t$, then $N(p, r(t))\subset N(q, 2t)$, where $r(t)= \arcsinh(\sinh(t)\sin\theta_0)$.
\end{lemma}
\begin{proof}
By Lemma \ref{distfromline}, $\arcsinh (x/y_0)= d(p, Y) <t$. The hyperbolic length of the horizontal segment in $\H^2$ from $p$ to $q$ is $x/y_0$. So,
$$d(p,q)<\frac{x}{y_0}<\sinh(t)$$
To show that $N(p, r(t)) \subset N(q, 2t)$, it is enough to show that $d(q, p)+r(t) <2t$.
\begin{align*}
    d(p, q)+r(t) &< \sinh (t) + \arcsinh\left(\sinh (t)\sin \theta_0\right)\\
    &\leq \sinh (t)(1+\sin \theta_0) \text { as } \arcsinh t \leq t \text{ for } t>0\\
    &\leq \sinh (t)\left(1+\frac{\sqrt{3}}{2}\right) \text{ as } 0<\theta_0\leq \pi/3\\
\end{align*}

As $\sinh(t)\leq t+t^3/5$ for $t\in(0,1/2)$ so we get,
$$
d(p, q)+ r(t) < t\left(1+\frac{t^2}{5}\right) \left(1+\frac{\sqrt{3}}{2}\right)
<t\left(1+\frac{1}{20}\right) \left(1+\frac{\sqrt{3}}{2}\right)< 2t
$$
\end{proof}

The following are some inequalities we shall repeatedly use:
\begin{lemma}\label{abound}
Let $a_0=\arcsinh(l_0\sin\theta_0/z_0)$ and $r(t)= \arcsinh(\sinh(t)\sin\theta_0)$. Let $\epsilon$ be the Margulis number for cusped complete hyperbolic manifolds. For any $t>0$, $r(t\,a_0)<t a_0$ and $a_0 <\sinh(a_0)<\epsilon<1$.
\end{lemma}
\begin{proof}
As $0<\theta_0 \leq \pi/3$ so $0<\tan \theta_0 \leq \sqrt{3}$. The function $g(m)=(\sqrt{m^2+2m}-m)/m$ is decreasing taking the value $\sqrt{3}-1 <1$ at $m=1$ and $v_{tet}>1$ so substituting the values of $a_0$, $l_0$ and $z_0$ and putting $m=1$ we get,
$$
a_0<\sinh(a_0) = \frac{l_0 \sin \theta_0}{z_0} = \frac{\epsilon (\sin \theta_0)^{4m+2}\sqrt{\tan \theta_0} \, g(m)}{4\sqrt{2m v_{tet}}} < \frac{3^{1/4}\epsilon}{4\sqrt{2}} < \epsilon
$$

The Margulis constant $\epsilon$ for cusped complete hyperbolic $3$-manifolds is less than $1$.  To see this, observe that a cusp neighbourhood in such a manifold can be expanded until its closure first touches itself on its boundary. The waist size of a cusp is the length of a shortest essential closed curve (that avoids the points of self-tangency) on the boundary of such a maximal cusp neighbourhood. Such a curve will correspond to a parabolic isometry in the fundamental group. Adams\cite{Ada2002} has shown that the waist size of the Figure 8 knot complement is $1$. By covering such a maximal cusp neighbourhood with a horoball in $\H^3$ centered at $\infty$ with horosphere boundary the plane $z=1$, we can see that there is a closed curve of length $2\arcsinh(1/2) \sim 0.9624$ that lies in the closure of the maximal cusp neighbourhood and touches a point of its boundary. And so the optimal Margulis number for the Figure 8 knot complement is less than $0.9625$. As $\epsilon$ is the infimum of the optimal Margulis numbers for all cusped complete hyperbolic $3$-manifolds so $\epsilon<1$.

And finally, the function $\arcsinh$ is strictly increasing so we get,
$$r(t\,a_0)=\arcsinh(\sinh(t\,a_0) \sin\theta_0)<\arcsinh(\sinh(t \, a_0))=t \, a_0
$$
\end{proof}

We are now in a position to give a lower bound on the systole length of $M$.
\begin{lemma}\label{mainlem3}
Let $M$ be a complete orientable cusped hyperbolic $3$-manifold. Let $\tau$ be a geometric ideal $\theta_0$-thick triangulation of $M$ with at most $m$ many $3$-simplexes. Let $\epsilon$ be the Margulis number for cusped complete orientable hyperbolic $3$-manifolds. Let $v_{tet}$ denote the volume of the regular ideal tetrahedron.
Then the systole length of $M$ is bounded below by $s_0(m, \theta_0)$ which is given by the following equations:
\begin{align*}
s_0=& \arcsinh(\sinh(a_0/4)\sin\theta_0)\\
a_0=& \arcsinh(l_0\sin\theta_0/z_0)\\
z_0=& \frac{\sqrt{2m\,v_{tet}\cot\theta_0}}{\epsilon\sin\theta_0}\\
l_0=& \frac{(\sin\theta_0)^{4m}(\sqrt{m^2+2m}-m)}{4m}
\end{align*}

\end{lemma}
\begin{proof}
Let $\gamma$ be a shortest closed geodesic of $M$. We need to show that the length of $\gamma$ is at least $s_0$. The thin part of $M$ consists of a pairwise disjoint union of cusp neighbourhoods $C_i$ and tubular neighbourhoods around short closed geodesics in $M\setminus \cup_i C_i$ called Margulis tubes. 

If $\gamma$ intersects the thick part of $M$ then some point of $\gamma$ has injectivity radius greater than or equal to $\epsilon/2$ and therefore length of $\gamma$ is at least $\epsilon$. By Lemma \ref{abound}, $s_0=r(a_0/4)<a_0/4<\epsilon$. So if $\gamma$ intersects the thick part of $M$ then its length $l(\gamma) >s_0$.

Cusp neighbourhoods have no minimal closed geodesics so $\gamma$ can not lie entirely in $\cup_i C_i$.  Assume that $\gamma$ lies in a Margulis tube. Interiors of simplexes are contractible so $\gamma$ intersects some face of $\tau$ in $M(z_0) \supset M \setminus \cup_i C_i$.

Let $p$ be a point of intersection of $\gamma$ with a face $F$ of $\tau$ in $M(z_0)$. Suppose that the distance between $p$ and an edge $E$ of $F$ is less than $a_0/4$. By Lemma \ref{abound}, $a_0<1$ so taking $t=a_0/4<1/2$ in Lemma \ref{sameht} there exists a point $q$ in $E\cap M(z_0)$ such that $N(p, s_0)=N(p, r(a_0/4)) \subset N(q, a_0/2)$. By Lemma \ref{edgeball}, $N(q, a_0/2)$ is an embedded ball in $M$. So $N(p,s_0)$ is also an embedded ball in $M$. If $d(p, \del F) \geq a_0/4$ then taking $t=a_0/4$ in Lemma \ref{faceball}, $N(p, r(a_0/4)/2)=N(p, s_0/2)$ is an embedded ball. In either case, because $\gamma$ is a closed geodesic through $p$ so $l(\gamma)>s_0$.
\end{proof}

Simplifying this bound results in a proof of Theorem \ref{mainthm3}.
\begin{proof}[Proof of Theorem \ref{mainthm3}]
Taking $t\in(0,1)$, $\sqrt{1+t} \geq 1+t/2-t^2/8$ and $t/2-t^2/8 > 0$ and we get the following identity
$$ 
\sqrt{\frac{\sqrt{1+t}-1}{2}} \geq \sqrt{\frac{\frac{t}{2}-\frac{t^2}{8}}{2}}
= \sqrt{\frac{t}{4}}\sqrt{1-\frac{t}{4}} >\frac{\sqrt{3t}}{4}
$$
As $\cosh(s)=2\sinh^2(s/2) +1$ and $\cosh^2(s)=1+\sinh^2(s)$ so 
$ 2\sinh^2(s/2) + 1 = \sqrt{1+\sinh^2(s)}$. By Lemma \ref{abound}, $\sinh(a_0)<1$ so putting $s=a_0$ and $t=\sinh^2(a_0)$ in the above identity we get,
$$
\sinh \left(\frac{a_0}{2}\right) = \sqrt{\frac{\sqrt{1+\sinh^2 (a_0)}-1}{2}} >\frac{\sqrt{3}}{4}\sinh(a_0)
$$
The $\sinh$ function is increasing so $\sinh(a_0/2)<\sinh(a_0)<1$. We repeat the step above with $s=a_0/2$ and $t=\sinh^2(a_0/2)$ to get
$$
\sinh \left(\frac{a_0}{4}\right) = \sqrt{\frac{\sqrt{1+\sinh^2 (a_0/2)}-1}{2}} >\frac{\sqrt{3}\sinh(a_0/2)}{4}>\frac{3}{16}\sinh(a_0)
$$
Let $g(m)= (\sqrt{m^2 + 2m}-m)/m$. Substituting the values of $l_0$ and $z_0$ we get,
$$
\sinh(a_0)=\frac{l_0\sin\theta_0}{z_0} = \frac{\epsilon (\sin\theta_0)^{4m+2}\sqrt{\tan\theta_0}\, g(m)}{4\sqrt{2m v_{tet}}} > \frac{\epsilon (\sin\theta_0)^{4m+5/2} g(m)}{4\sqrt{2m v_{tet}}}
$$

Using the inequality $\sqrt{1+t} \geq 1+t/2-t^2/8$ again with $t=2/m$ we get the following lower bound for $g(m)$ as $m \geq 1$.
$$ 
g(m)= \sqrt{1+2/m}-1\geq \frac{2m-1}{2m^2} \geq \frac{1}{2m}$$

Therefore,
$$\sinh\left(\frac{a_0}{4}\right)\sin\theta_0 > \frac{3}{16}\sinh(a_0)\sin\theta_0>\frac{3\epsilon(\sin\theta_0)^{4m+7/2}} {128m\sqrt{2m\, v_{tet}}}$$
The $\sinh$ function is increasing so $\sinh(a_0/4)\sin\theta_0 < \sinh(a_0)<1$. And $\arcsinh(t) > t/2$ for $t< 4$ so taking $t=\sinh(a_0/4)\sin\theta_0$ we get
$$
s_0=\arcsinh(\sinh(a_0/4)\sin\theta_0) > \frac{\sinh(a_0/4)\sin\theta_0}{2} >\left(\frac{3\epsilon}{256\sqrt{2v_{tet}}}\right) \frac{(\sin\theta_0)^{4m+7/2}}{m\sqrt{m}}
$$

As $\epsilon \geq 0.29$ and $v_{tet}<1.02$ so $3\epsilon/\sqrt{2v_{tet}}>1/2$. Therefore using Lemma \ref{mainlem3} we can conclude that $2^{-9} (\sin\theta_0)^{4m+7/2} m^{-3/2}$ is a lower bound for the systole length of $M$.
\end{proof}

\subsection{Bounded intersection of ideal triangulations}
The intersection of simplexes of ideal geometric triangulations $\tau_1$ and $\tau_2$ of $M$ give a common polytopal subcomplex $\tau_1 \cap  \tau_2$ which may have material (non-ideal) vertices. In this subsection we calculate an explicit bound on the number of polytopes in $\tau_1 \cap \tau_2$ (Theorem \ref{mainsubthm}).

\begin{lemma}\label{sectorvol}
Let $\tau$ be a $\theta_0$-thick geometric ideal triangulation of $M$. Let $\Delta$ be an ideal tetrahedron in $\tau$ and let $p\in \del \Delta \cap M(z_0)$. Let $\lift{\Delta}$ be a lift of $\Delta$ to $\H^3$ and let $\lift{p}$ be a lift of $p$ in $\lift{\Delta}$. Let $r_0=\arcsinh(\sinh(a_0/2)\sin\theta_0)$. Then $N(\lift{p}, r_0) \cap \lift{\Delta}$ has volume at least $\theta_0/2\pi \cdot vol(B(r_0))$, where $vol(B(r_0))$ denotes the volume of a ball of radius $r_0$ in $\H^3$.
\end{lemma}

\begin{proof}
Let $F$ be a face of $\Delta$ containing $p$ and let $\lift{p}\in \lift{F}$ which is a lift of $F$ to a face of $\lift{\Delta}$. If $d(\lift{p}, \partial \lift{F}) \geq a_0/2$ then putting $t=a_0/2$ in Lemma \ref{faceball}, we get $N(\lift{p}, r(a_0/2))=N(\lift{p}, r_0)$ is an embedded ball in $instar(\lift{F})$. Simplexes in $\lift{\tau}$ are uniquely determined by their ideal vertices on $\del \H^3$ so $star(\lift{F})$ is the union of two tetrahedra $\lift{\Delta}$ and $\lift{\Delta}'$ identified along $\lift{F}$. Hence $\lift{F}$ divides $N(\lift{p}, r_0)$ into congruent halves one of which lies entirely in $\lift{\Delta}$. Therefore  $N(\lift{p}, r_0) \cap \lift{\Delta}$ has volume equal to $1/2 \cdot vol(B(r_0)) \geq \theta_0/2\pi \cdot vol(B(r_0))$ as $\theta_0 \leq\pi/3$.

If for some edge $\lift{E}$ of $\lift{F}$, $d(\lift{p}, \lift{E}) <a_0/2$ then choose a lift $\lift{\tau}$ where $\lift{E}$ is a vertical geodesic. Let $\lift{q}$ be a point on $\lift{E}$ at the same height as $\lift{p}$. By Lemma \ref{abound}, $a_0/2<1/2$ and so putting $t=a_0/2$ in Lemma \ref{sameht}, $N(\lift{p}, r_0) \subset N(\lift{q}, a_0)$. As $\lift{p} \in \H^3\setminus H(z_0)$ so is $\lift{q}$ and by Lemma \ref{edgeball}, $N(\lift{q}, a_0)$ is a ball in $instar(\lift{E})$. We claim that there exists a sector of the ball $N(\lift{p}, r_0)$ with dihedral angle $\theta_0$ that lies entirely inside $\lift{\Delta}$.

\begin{figure}
\centering
\def\svgwidth{0.4\columnwidth}
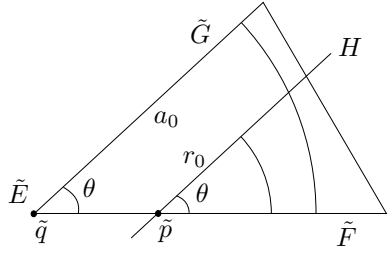
\caption{Faces $\lift{F}$ and $\lift{G}$ meet along the edge $\lift{E}$, $H$ is a plane parallel to $\lift{G}$ through $\lift{p}$ and $N(\lift{p}, r_0) \subset N(\lift{q}, a_0)$, so the sector $S(\lift{p}, r_0, \theta) \subset S(\lift{q}, a_0, \theta)= N(\lift{q}, a_0) \cap \lift{\Delta}$.}\label{sectorvol_fig}
\end{figure}

Let $\lift{G}$ be the other face of $\lift{\Delta}$ that contains $\lift{E}$. See Figure \ref{sectorvol_fig} for a cross-section by a horosphere centered at $\infty$. Let $\theta \geq \theta_0$ be the dihedral angle between $\lift{F}$ and $\lift{G}$ at $\lift{E}$. The ball $N(\lift{q}, a_0)$ lies in the interior of $star(\lift{E})$ so $\lift{\Delta} \cap N(\lift{q}, a_0)$ is the sector $S(\lift{q}, a_0, \theta)$ of a ball centered at $\lift{q}$ of radius $a_0$ and sectorial angle $\theta$. Let $H$ be a vertical geodesic plane through $\lift{p}$ parallel to $\lift{G}$. $H$ divides $S(\lift{q}, a_0, \theta)$ into two regions one of which intersects $\lift{E}$ and one which does not. Let $\delta$ denote the region which does not intersect $\lift{E}$. We have $N(\lift{p}, r_0) \subset N(\lift{q}, a_0)$ so $N(\lift{p}, r_0) \cap \delta$ is the sector $S(\lift{p}, r_0, \theta) \subset S(\lift{q}, a_0, \theta)$. As $\theta\geq\theta_0$ so $S(\lift{p}, r_0, \theta_0) \subset N(\lift{p}, r_0) \cap \lift{\Delta}$ and the volume of $S(\lift{p}, r_0, \theta_0)$ is $\theta_0/(2\pi) \cdot vol(B(r_0))$. 
\end{proof}

Let $\tau$ be a geometric ideal $\theta_0$-thick triangulation of $M$. In Lemma \ref{edgeball} we calculated a lower bound on the distance between the edges of $\tau$ in the thick part of $M$. Using this bound we now give an upper bound on the number of connected components in the intersection of the edge set of $\tau$ and a tetrahedron $\Delta$ of another  geometric ideal triangulation $\tau'$.

\begin{lemma}\label{finiteint}
Let $\tau$ and $\tau'$ be geometric ideal $\theta_0$-thick triangulations of $M$. Let $pr: \H^3 \to M$ be a covering map. Let $\lift{\tau}$ and $\lift{\tau}'$ denote lifts of $\tau$ and $\tau'$ with respect to $pr$ to triangulations of $\H^3$. Let $\Delta$ be an ideal tetrahedron of $\tau'$ and let $\lift{\Delta}$ denote a lift of $\Delta$ in $\lift{\tau}'$.  Let $E(\lift{\tau})$ be the set of edges of $\lift{\tau}$. Let $r_0=\arcsinh(\sinh(a_0/2)\sin\theta_0)$, let $v_{tet}$ denote the volume of the regular ideal tetrehedron and let $n$ be the number of components of $E(\lift{\tau}) \cap \lift{\Delta}$. Then  
$$n \leq \frac{2\pi \, v_{tet}}{\theta_0 \,vol(B(r_0/2))}$$ 
\end{lemma}
\begin{proof}
Let $m_1$ and $m_2$ be the number of ideal tetrahedra in $\tau$ and $\tau'$ respectively. Put $m=m_1+m_2$ in Remark \ref{constants} to obtain $z_0$ larger than  the $\tau$-normalised cusp heights and the $\tau'$-normalised cusp heights of all cusps of $M$. Let $pr_i$ be a covering projection that is normalised with respect to the $i$-th cusp. 

Let $H(z_0) =\{(x, y, z) \in \H^3: z>z_0\}$ and let $X_{z_0} = \R^2 \times \{z_0\}$. Let $C_i(z_0) = pr_i(H(z_0))$ and let $cl(C_i(z_0))$ denote its closure in $M$. Let $C(z_0) = \cup_{i=1}^k C_i (z_0)$ and let $M(z_0) = M \setminus C(z_0)$ be its complement. $M(z_0)$ is a compact hyperbolic manifold with tori boundary $T_i(z_0) = pr_i (X_{z_0})$, $i=1, ..., k$. All components of $\del M(z_0)$ are flat in the induced metric and both $\tau$ and $\tau'$ induce Euclidean triangulation on $\del M$.

We first claim that each connected component $\sigma$ of $E(\tau) \cap \Delta$ intersects $\del \Delta \cap M(z_0)$. Let $E$ be the edge of $\tau$ that contains $\sigma$. If $\sigma$ does not intersect $\del \Delta$ then both its end points are ideal vertices, in which case $\sigma=E$. It has a lift $\lift{E} \subset \lift{\Delta}$. Both the ideal vertices of $\lift{E}$ are also ideal vertices of $\lift{\Delta}$ so $\lift{E}$ is an edge of $\lift{\Delta}$. Consequently $E$ is an edge of $\Delta$. All edges of $\Delta$ intersect $M(z_0)$, so $\sigma \cap \del \Delta \cap M(z_0)= E\cap M(z_0)$ is non-empty.  Therefore we may assume that $\sigma$ intersects $\del \Delta$. Assume that $\sigma$ intersects $\del \Delta$ in $C(z_0)$, i.e., for some face $F$ of $\Delta$ there exists  $p \in \sigma \cap F \cap C_i(z_0)$ for some $1 \leq i \leq k$. The triangulations induced by $\tau$ and $\tau'$ in $cl(C_i(z_0))$ are the cones over $\tau\cap T_i(z_0)$ and $\tau'\cap T_i(z_0)$ respectively.  Each edge of $\tau$ and each face of $\Delta$ which meets the cusp $C_i(z_0)$ is orthogonal to the cusp torus $T_i(z_0)$. So there exists a point $\bar{p} \in T_i(z_0)$ and a geodesic ray $[\bar{p}, c_i)$ in $cl(C_i(z_0))$ from $\bar{p}$ orthogonal to $T_i(z_0)$ which is the connected component of $E \cap cl(C_i(z_0))$ containing $p$. The point $p$ lies in $[\bar{p}, c_i) \cap F$ so $[\bar{p}, c_i)$ is a subset of $F \cap cl(C_i(z_0))$. Therefore $\sigma \supset [\bar{p}, c_i)$ and in particular $\bar{p} \in \sigma \cap \del \Delta \cap M(z_0)$.

Lifting to $\H^3$ under $pr$, each connected component $\sigma$ of $E(\lift{\tau}) \cap \lift{\Delta}$ intersects $\del \lift{\Delta}\cap pr^{-1}(M(z_0))$. For each such component $\sigma$ choose a point $p_\sigma \in \sigma \cap \del \lift{\Delta}\cap pr^{-1}(M(z_0))$ and let $A_\sigma=N(p_\sigma, r_0/2) \cap \lift{\Delta}$. Note that as $\arcsinh$ is an increasing function so $r_0 = \arcsinh(\sinh(a_0/2) \sin\theta_0) <a_0/2$.

We claim that the collection of sets $\{A_\sigma : \sigma$ \mbox{ is a component of } $E(\lift{\tau}) \cap \lift{\Delta}\}$ are pairwise disjoint subsets of $\lift{\Delta}$. Suppose there exist different components $\sigma_0$ and $\sigma_1$ of $E(\lift{\tau})  \cap \lift{\Delta}$ such that the corresponding $A_{\sigma_0}$ intersects $A_{\sigma_1}$. Then for $i=0,1$, there exist points $p_i\in \sigma_i \cap \del \lift{\Delta} \cap pr^{-1}(M(z_0))$ with $d(p_0, p_1) <r_0$. Assume that $\sigma_i$ lies in an edge $E_i$ of $\lift{\tau}$. By Lemma \ref{edgeball}, $N(p_0, r_0) \subset N(p_0, a_0)$ is an embedded ball in $instar(E_0)$ and as $E_0$ is the only edge of $\lift{\tau}$ that intersects $instar(E_0)$ so $E_0=E_1$. Let $E=E_0=E_1$. The points $p_0$ and $p_1$ lie on different components of $E(\lift{\tau}) \cap \lift{\Delta}$, so the edge segment $E_{[p_0,p_1]}$ can not lie entirely in $\lift{\Delta}$. But as $\lift{\Delta}$ is convex so there exists some geodesic from $p_0$ to $p_1$ in $\lift{\Delta}$. We therefore end up with two distinct geodesics between a pair of points in $\H^3$, which is a contradiction.

By Lemma \ref{sectorvol}, the volume of $A_\sigma$ is at least $v=\theta_0\, vol(B(r_0/2))/(2\pi)$. Assume there are $n$ many components $\sigma_i$ of $E(\lift{\tau}) \cap \lift{\Delta}$. The corresponding sets $A_{\sigma_i}$ are disjoint so $nv \leq vol(\lift{\Delta})\leq v_{tet}$. This gives, $n \leq v_{tet}\, 2\pi/(\theta_0 vol(B(r_0/2)))$ as required. 
\end{proof}

\begin{figure}
\centering
\def\svgwidth{0.4\columnwidth}
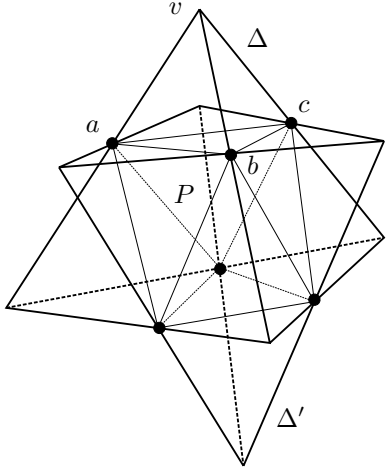
\caption{Tetrahedra $\Delta$ and $\Delta'$ drawn with solid lines intersecting in an octahedron $P$ with one vertex on each edge of $\Delta$.}\label{badpoly}
\end{figure}

We shall use the next lemma to argue that the number of polytopes in the polytopal complex $\tau \cap \Delta$ depends on the number of components of $E(\tau) \cap \Delta$.

\begin{lemma}\label{edgesub}
Let $\Delta$  be a hyperbolic ideal tetrahedron in $\H^3$ and let $\tau$ be an ideal triangulation of $\H^3$. There exists at most one $3$-dimensional polytope $P$ in the polytopal complex $\Delta \cap \tau$ with the property that no edge of $P$ lies in an edge of $\Delta$ or in an edge of $\tau$.
\end{lemma}
\begin{proof}
Let $\Delta'$ be a tetrahedron of $\tau$ and let $P=\Delta \cap \Delta'$ be a $3$-polytope in the polytopal complex $\Delta \cap \tau$. We shall first prove that $P$ has exactly $6$ vertices, one on each edge of $\Delta$, as in Figure \ref{badpoly}. And then we shall prove that there is a unique such polytope in the polytopal complex $\Delta \cap \tau$.

 Faces of $P$ are subsets of faces of $\Delta$ and $\Delta'$, so edges of $P$ are subsets of the intersection of faces of $\Delta$ and $\Delta'$. If no edge of $P$ lies in an edge of $\Delta$ or of $\Delta'$ then every edge of $P$ lies in the intersection of a face of $\Delta$ with a face of $\Delta'$. Let $v$ be a vertex of $P$. If we list the faces of $P$ that meet at $v$ in a clockwise fashion, they alternate between faces of $P$ that lie in faces of $\Delta$ and faces of $P$ that lie in faces of $\Delta'$.
So at least two faces of $\Delta$ and at least two faces of $\Delta'$ meet at every vertex of $P$. Assume that $\{v\}= F_1 \cap F_2 \cap G_1 \cap G_2$ where $F_1$ and $F_2$ are faces of $\Delta$ and $G_1$ and $G_2$ are faces of $\Delta'$. Let $e_1=F_1 \cap F_2$ and $e_2=G_1 \cap G_2$ be edges of $\Delta$ and $\Delta'$ respectively so that $v=e_1 \cap e_2$. If $w$ is another vertex of $P$ which lies on $e_1$ then as $P$ is convex so the edge segment $e_1|_{[v, w]}$ lies in $P$ and as $P \subset \Delta$ so $e_1|_{[v,w]}$ is in fact an edge of $P$. This contradicts the fact that no edge of $P$ lies in an edge of $\Delta$.

Therefore every vertex of $P$ lies on an edge of $\Delta$ and at most one vertex of $P$ lies on any edge of $\Delta$. The simplex $\Delta$ has $6$ edges so $P$ has at most $6$ vertices, each on a distinct edge of $\Delta$. And as at least $4$ faces meet at each vertex of $P$, the degree of each vertex, i.e., the number of edges of $P$ that meet at the vertex, is at least $4$.

$P$ is a $3$-polytope so it has at least $4$ vertices. If $P$ has $4$ vertices then it is a tetrahedron so each vertex has degree $3$. If $P$ has $5$ vertices, then the degree of each vertex must be exactly $4$. If $a, b, c, d$ are the vertices adjacent to a vertex $e$, then either they are coplanar and $[abcd]$ is a quadrilateral face of $P$ or they form two triangles, say $[abc]$ and $[bcd]$. In the first case $a, b, c, d$ all end up with degree $3$, in the second case $a$ and $d$ have degree $3$. The degree of each vertex of $P$ is at least $4$ so we can conclude that $P$ must have $6$ vertices, one on each edge of $\Delta$. See Figure \ref{badpoly} for an example of this exceptional polytope.

Let $v$ be a vertex of $\Delta$ and let $a, b, c$ be the vertices of $P$ that lie on the edges of $\Delta$ that contain $v$. Let $H$ be the geodesic plane in $\H^3$ containing $a$, $b$ and $c$. Then $H$ separates $v$ from the edges of $\Delta$ that do not contain $v$. No vertex of $P$ lies on the side of $H$ containing $v$ so $H \cap P$ is the triangle $[abc]$ which is the convex hull of $a, b, c$. Therefore $[abc]$ is a triangular face of $P$ and we call it a normal triangle of $P$ with respect to vertex $v$. 

Suppose there are two $3$-polytopes $P_1$ and $P_2$ in $\Delta \cap \tau$ with $6$ vertices, one on each edge of $\Delta$. Let $t_1$ and $t_2$ be the normal triangles of $P_1$ and $P_2$ with respect to vertex $v$ of $\Delta$. If $t_1$ intersects $t_2$ then the interiors of $P_1$ and $P_2$ intersect, which is a contradiction as they are both polytopes of a polytopal complex. If $t_1$ and $t_2$ are parallel then assume that $t_1$ is closer to $v$ than $t_2$. As $t_2$ separates $t_1$ and the edges of $\Delta$ not containing $v$, so in particular it separates the vertices of $t_1$ from the vertices of $P_1$ that lie on the edges of $\Delta$ not containing $v$. And so again, the interiors of $P_1$ and $P_2$ intersect. Therefore there is at most one polytope $P$ in $\Delta \cap \tau$ with the property that no edge of $P$ lies in an edge of $\Delta$ or $\tau$.
\end{proof}

We are finally in a position to prove the main theorem of this section:
\begin{theorem}\label{mainsubthm}
Let $M$ be an orientable complete cusped hyperbolic $3$-manifold. Let $\tau_1$ and $\tau_2$ be geometric ideal $\theta_0$-thick triangulations of $M$ with at most $m_1$ and $m_2$ many $3$-simplexes respectively. Let $m=m_1+ m_2$. Let $v_{tet}$ denote the volume of the regular hyperbolic ideal tetrahedron and let $\epsilon$ be the Margulis number for cusped orientable hyperbolic $3$-manifolds. The total number of $3$-polytopes in the polytopal complex $\tau_1 \cap \tau_2$ is bounded above by 
$$f(m, \theta_0)=\left( \frac{4\pi v_{tet}}{\theta_0^2 (sinh(r_0) - r_0)}+1\right)m$$
where 
\begin{align*}
r_0 =& \arcsinh(\sinh(a_0/2)\sin\theta_0)\\
a_0=& \arcsinh(l_0\sin\theta_0/z_0)\\
z_0=& \sqrt{2m\,v_{tet}\cot\theta_0}/(\epsilon\sin\theta_0)\\
l_0=& (\sin\theta_0)^{4m}(\sqrt{m^2+2m}-m)/(4m)
\end{align*}
\end{theorem}

\begin{proof}
Let $d$ denote the maximum number of 3-simplexes of $\lift{\tau}_1$ that share an edge. The lift $\lift{\tau}_1$ is $\theta_0$-thick, so $d \leq 2\pi/\theta_0$. Let $\lift{\Delta}$ be the lift of a tetrahedron $\Delta$ of $\tau_2$. Let $E(\lift{\tau}_1)$ denote the edge set of $\lift{\tau}_1$. The number of components $n$ of  $E(\lift{\tau}_1)\cap \lift{\Delta}$ is bounded above as in Lemma \ref{finiteint}. So the number of 3-polytopes in $\lift{\tau}_1\cap \lift{\Delta}$ which have an edge that lies inside an edge of $\lift{\tau}_1$ is bounded above by $dn$. The covering projection restricts to an isometry from $\lift{\Delta}$ to $\Delta$, so the number of polytopes in $\tau_1 \cap \Delta$ that have an edge which lies inside an edge of $\tau_1$ is also bounded above by $dn$. Varying $\Delta$ over all 3-simplexes of $\tau_2$, the total number of polytopes of $\tau_1 \cap \tau_2$ which have an edge that lies in an edge of $\tau_1$ is bounded above by $dnm_2 \leq (2\pi/\theta_0)\, nm_2$. Similarly the total number of polytopes of $\tau_1 \cap \tau_2$ which have an edge that lies in an edge of $\tau_2$ is bounded above by $(2\pi/\theta_0) \,nm_1$. 

Each polytope of $\tau_1 \cap \tau_2$ has an edge which lies in either an edge of $\tau_1$ or $\tau_2$, barring the exceptional polytopes described in Lemma \ref{edgesub}. Each tetrahedron of $\tau_1$ and of $\tau_2$ has at most one such exceptional polytope, so in total there are at most $\min(m_1, m_2)$ many of them. Therefore substituting the bound for $n$ obtained in Lemma \ref{finiteint}, the total number of polytopes in $\tau_1 \cap \tau_2$ is bounded above by 
$$
(2\pi/\theta_0) n(m_1+m_2) + \min (m_1, m_2)  \leq  ((2n\pi/\theta_0)+1)m \leq\left( \frac{(2\pi)^2 v_{tet}}{\theta_0^2 \, vol(B(r_0/2))}+1\right)m
$$
The hyperbolic volume of a ball of radius $r_0/2$ is $\pi(sinh(r_0) - r_0)$. Substituting this for $vol(B(r_0/2))$ gives the required bound.

\end{proof}

The proof of Theorem \ref{mainthm2} now trivially follows from the following result of \cite{KalPha2}. See Definition \ref{barydefn} for a definition of derived subdivision.

\begin{theorem}[Theorem 1.2 of \cite{KalPha2}]\label{geompach}
Let $K_1$ and $K_2$ be geometric simplicial triangulations (possibly with material vertices) of a cusped hyperbolic manifold which have a common geometric subdivision (with finitely many simplexes). Then for some $s\in \N$, the $s$-th derived subdivisions $\beta^s K_1$ and $\beta^s K_2$ are related by geometric Pachner moves.
\end{theorem}

\begin{proof}[Proof of Theorem \ref{mainthm2}]
Let $K_1 = \beta^2 \tau_1$ and let $K_2=\beta^2 \tau_2$ be the second derived subdivision of $K_1$ and $K_2$. They are both geometric simplicial triangulations, i.e., each simplex of $K_i$ is determined by its (material and ideal) vertices. Each tetrahedron $\Delta$ of $\tau_i$ is subdivided into $(4!)^2$ tetrahedra in $K_i$. Let $\Delta_1$ and $\Delta_2$ be tetrahedra of $\tau_1$ and $\tau_2$ and let $P$ be a connected component of $\Delta_1 \cap \Delta_2$. The polytope $P$ is convex so the covering projection map restricts to an isometry from $\lift{P}$ to $P$. There exist lifts $\lift{\Delta}_i$ of $\Delta_i$ such that  $\lift{P}=\lift{\Delta}_1 \cap \lift{\Delta}_2$. So the number of polytopes in $\beta^2 \lift{\Delta}_1 \cap \beta^2 \lift{\Delta}_2$ is at most $(4!)^4$. Consequently the number of $3$-polytopes in $K_1 \cap K_2$ is at most $(4!)^4$ times the number of $3$-polytopes in $\tau_1 \cap \tau_2$. By Theorem \ref{mainsubthm}, $\tau_1 \cap \tau_2$ has finitely many $3$-polytopes, so $\beta(K_1 \cap K_2)$ is a (finite) common geometric simplicial subdivision of both $K_1$ and $K_2$. 

We can now apply Theorem \ref{geompach} to obtain a sequence of Pachner moves through geometric triangulations between $\beta^s K_1=\beta^{s+2} \tau_1$ and $\beta^s K_2 = \beta^{s+2} \tau_2$. In dimension 3, it is easy to see that derived subdivisions of geometric triangulations can be realised by Pachner moves through geometric triangulations (see for example Lemma 2.11 of \cite{IzmSch}). So $\tau_1$ and $\tau_2$ are related by Pachner moves through geometric triangulations, via $\beta^{s+2} \tau_1$ and $\beta^{s+2} \tau_2$.
\end{proof}

\section{Bound on Pachner moves}
In this section we use the bound on the number of polytopes in a common polytopal 
subdivision calculated in the previous section to prove Theorem \ref{mainthm1}. The triangulations we shall consider in this section may be non-ideal, i.e., they may have material vertices. We call the topological triangulation of a hyperbolic manifold geometric if the relative interior of every $n$-simplex is a totally geodesic $n$-disk. The combinatorial techniques we shall use are from previous work by Phanse and the first author\cite{KalPha} with tighter bounds calculated here for dimension $3$. 
\begin{definition}
Let $pr: \H^3 \to M$ be a covering map. Let $\tau$ be a topological triangulation of $M$ possibly with ideal and material vertices. Let $D$ be a sub-complex of $\tau$ such that its lift $\lift{D}$ is a simplicially triangulated closed $3$-ball subcomplex of $\lift{\tau}$ in $\H^3\cup \del \H^3$. Let $D'$ be a triangulated closed 3-ball subcomplex of $\del \Delta^4$ and let $\phi: D' \to \lift{D}$ be a simplicial isomorphism. A bistellar or Pachner move on $\tau$ consists of removing $D$ and replacing it with $D'$ attached along the boundary $pr(\phi(\del D'))$. See Figure  \ref{Pachnerfig} for the four possible Pachner moves (in dimension $3$).
\end{definition}

\begin{figure}
\centering
\def\svgwidth{1\columnwidth}
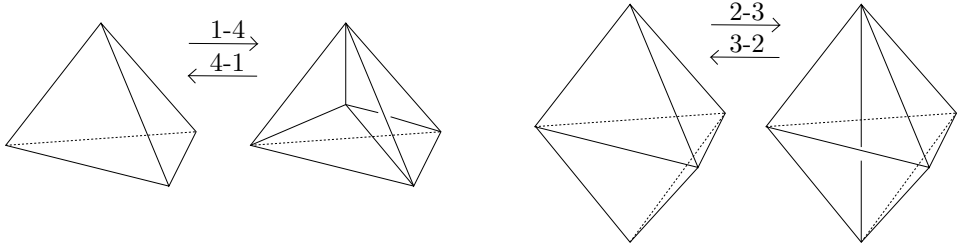
\caption{The two pairs of Pachner moves.}\label{Pachnerfig}
\end{figure}

\begin{definition}
Let $K$ be a simplicial topological triangulation of a closed $n$-ball. A shelling of $K$ is an ordering $\Delta_1$, $\Delta_2$, ..., $\Delta_k$ of the $n$-simplexes of $K$ such that for $1 <j\leq k$, $\Delta_j \cap(\cup_{i=1}^{j-1} \Delta_i)$ is an $(n-1)$-disk subcomplex of $\del \Delta_j$. We say $K$ is shellable if it has a shelling sequence.
Let $K'$ be a simplicial topological triangulation of an $n$-sphere. We say $K'$ is shellable if for some $n$-simplex $\Delta_0$ of $K'$, $K'\setminus \Delta_0$ is shellable. Let $L$ be the subcomplex of a triangulation of $M$. We say $L$ is shellable if there exists a lift of $L$ to $\H^3$ which is shellable.
\end{definition}

It is easy to see that $2$-polytopes are shellable. Higher dimensional polytopes though may not be shellable. Rudin\cite{Rud} gave an example of a Euclidean subdivision of a Euclidean 3-simplex which is not shellable. Lickorish\cite{Lic2} has given a family of unshellable topological triangulations of a 3-sphere. The main result we shall use in this section is by Adiprasito and Benedetti\cite{AdiBen} who showed that a derived subdivision of a Euclidean triangulation of a convex 3-polytope is always shellable. 

Shellable $n$-balls are 'starrable' i.e., a shellable topological triangulation of the $n$-ball can be changed to the cone over the boundary of the ball by Pachner moves (through topological triangulations). For the sake of completeness we give here a proof of Lemma 5.7 of \cite{Lic} in dimension $3$.

\begin{lemma}[Lemma 5.7 of \cite{Lic}] \label{shellable}
Let $K$ be a shellable triangulation of a $3$-ball $B$ with $r$ many $3$-simplexes, then $K$ is related to $p \star \del K$ by a sequence of $r$ Pachner moves, where $p\in int(B)$.
\end{lemma}
\begin{proof}
We prove this by induction on the number $r$ of $3$-simplexes of $K$. When $r=1$, then $K$ is a $3$-simplex and a single 1-4 Pachner move introducing the new vertex $p$ changes $K$ to $p\star \del K$. 

Let $\Delta_1, ..., \Delta_r$ be a shelling ordering for the $3$-simplexes of $K$. As $K'=\cup_{i=1}^{r-1} \Delta_i$ is a triangulated $3$-ball with $r-1$ many $3$-simplexes so by induction $K'$ is related to $p* \del K'$ by $r-1$ many Pachner moves. Let $\Delta_r=[abcd]$. Let $D=\del \Delta_r \cap \del K'$ and let $D'=\del \Delta_r \setminus int(D)$ be $2$-disk subcomplexes of $\Delta_r$

There are three possibilities for $D$. If $D$ is a $2$-simplex say $[bcd]$ then $D' = [abc]\cup[abd] \cup [acd]$. And a 2-3 Pachner move changes $\Delta_r \cup p\star D$ to $p\star D'$. If $D$ is a union of two $2$-simplexes say $[abc] \cup [bcd]$ then $D'=[abd]\cup[acd]$ and a 3-2 Pachner move change $\Delta_r \cup p\star D$ to $p\star D'$. And lastly if $D$ is the union of three $2$-simplexes say $[abc]\cup[abd]\cup [acd]$ then $D'=[bcd]$ and a 4-1 Pachner move changes $\Delta_r \cup p\star D$ to $p\star D'$. So exactly one Pachner move is needed to change $(p \star \del K') \cup \Delta_r$ to $p\star \del K$. So in all, we need $r$ Pachner moves to change $K$ to $p\star \del K$.
\end{proof}

\begin{definition}\label{barydefn}
Let $K$ be the geometric triangulation (possibly with material vertices) of a hyperbolic manifold $M$. Let $\alpha K$ be a geometric subdivision of $K$. Let $\beta^\alpha_r K$ be the geometric subdivision of $K$ such that, if $A$ is a simplex in $K$ and $dim (A) \leq r$, then $\beta^\alpha_r A = \alpha A$ and if $dim (A) > r$ then $\beta^\alpha_r A = b(A) \star \beta^\alpha_r \del \alpha A$ for some point $b(A)$ in the relative interior of $A$ (with $b(A)=A$ if $A$ is a vertex), i.e. it is subdivided as the geometric cone on the already defined subdivision of its boundary. In other words, if $K^{(r)}$ denoted the  the $r$-skeleton of $K$, then $\beta^\alpha_r K^{(r)}=\alpha K^{(r)}$ and for $s>r$, any $s$-simplex of $K$ is inductively subdivided as the the cone over its boundary. Observe that $\beta^\alpha_3 K$ is $\alpha K$ while $\beta^\alpha_0 K = \beta K$ is called the derived subdivision of $K$. When $\alpha K = K$, we denote $\beta^\alpha_r K$ by $\beta_r K$ and call it a partial derived subdivision. See Figure 4 of \cite{KalPha} for an example.
\end{definition}

Given $\theta_0$-thick geometric ideal triangulations $K_1$ and $K_2$ of $M$ with at most $m_1$ and $m_2$ many ideal tetrahedra, we obtained a bound $f$ in Theorem \ref{mainsubthm} on the number of 3-polytopes in the polytopal complex $K_1 \cap K_2$. Its derived subdivision $K'=\beta(K_1 \cap K_2)$ is then a common geometric subdivision of $K_1$ and $K_2$. In this section, we bound the number of Pachner moves needed to change $K_i$ to $\beta K'$, which leads to a proof of Theorem \ref{mainthm1}.

Let $pr: \H^3 \to M$ be a covering projection and let $K$ be a geometric ideal simplicial triangulation of $M$. Then for any simplex $A$ of $K$ and lift $\lift{A}$ of $A$ in $\lift{K}$, by definition $pr$ restricts to an isometry from $star(\lift{A}, \lift{K})$ to $star(A, K)$. Even when $K$ is not simplicial, such a property almost holds for partial derived subdivisions:

\begin{figure}
\centering
\def\svgwidth{0.8\columnwidth}
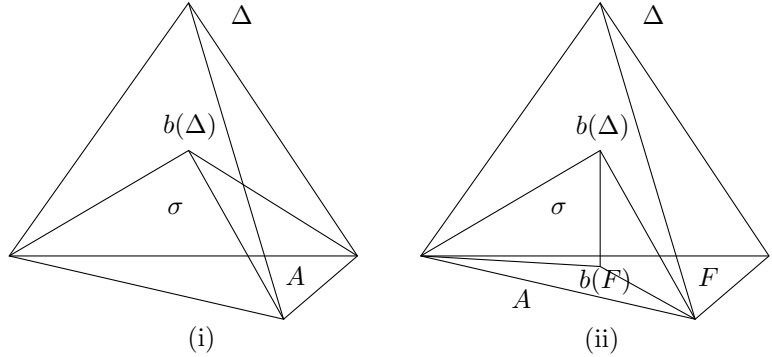
\caption{(i) $A$ is a face of $\Delta$ and $\sigma$ a 3-simplex in $star(A, \beta_2 \Delta)$ (ii) $A$ is an edge of $\Delta$ and $\sigma$ a 3-simplex in $star(A, \beta_1 \Delta)$ }\label{pder}
\end{figure}

\begin{lemma}\label{starballs}
Let $pr:\H^3 \to M$ be a covering projection. Let $K$ be a geometric ideal triangulation of $M$ and let $A$ be an $r$-simplex of $K$. Let $\lift{A}$ be a lift of $A$ in the lifted geometric ideal triangulation $\lift{K}$ of $\H^3$. Then $A$ is also an $r$-simplex of $\beta_r K$ and $pr$ restricts to an isometry from $instar(\lift{A}, \beta_r \lift{K})$ to $instar(A, \beta_r K)$.
\end{lemma}
\begin{proof}
Let $\sigma$ be a 3-simplex in $star(A, \beta_r K)$. We will first show that $\sigma \setminus \del A$ is contractible in $M$. Let $\Delta$ be the $3$-simplex of $K$ which contains $\sigma$. When $r=3$ then $A$ is a tetrahedron, $\beta_3 K=K$ and $\sigma=A=star(A, \beta_3 K)$. So $\sigma \setminus \del A = int(A)$, which is contractible. When $r=2$, $A$ is a face and $\sigma = A \star b(\Delta) \subset \Delta$, where $b(\Delta)$ denotes a point in the interior of $\Delta$, as in Figure \ref{pder}(i). So $\sigma \setminus \del A$ is a convex subset of the embedded ball $int(\Delta) \cup relint(A)$ and is therefore contractible. (The notation $relint(A)$ denotes the relative interior of the simplex $A$). When $r=1$, $A$ is an edge and $\sigma=A \star b(F) \star b(\Delta)$ where $F$ is a face of $\Delta$ and $b(F)$ denotes a point in the interior of $F$, as in Figure \ref{pder}(ii). The boundary of $A$ is just the ideal vertices, so $\sigma\setminus \del A=\sigma$ is a convex subset of the embedded ball $int(\Delta) \cup  relint(F) \cup relint(A)$ so $\sigma\setminus \del A$ is contractible. \\

Let $\lift{A}$ be a lift of the $r$-simplex $A$ in $\lift{K}$. The ideal vertices of a simplex of $\lift{K}$ uniquely determine the simplex. And therefore $star(\lift{A}, \beta_r \lift{K})$ is a closed ball.

Suppose there exist points $\lift{q}_1$ and $\lift{q}_2$ in $instar(\lift{A}, \beta_r \lift{K})$ which project down to the same point $q$ in $M$. Assume that $q$ lies in a 3-simplex $\sigma$ of $star(A,\beta_r K)$. The boundary $\del \lift{A}$ is disjoint from $instar(\lift{A}, \beta_r K)$ so $\lift{q}_1$ and $\lift{q}_2$ do not lie in $\del \lift{A}$. The simplex $\sigma\setminus \del A$ is contractible in $M$ so lifts of $\sigma \setminus \del A$ are disjoint in $\H^3$. All $3$-simplexes of $star(\lift{A}, \beta_r\lift{K})$ have a common intersection at $\lift{A}$ so exactly one of the $3$-simplexes in $star(\lift{A}, \beta_r\lift{K})$ projects down to $\sigma$. Therefore $\lift{q}_1$ and $\lift{q}_2$ lie in the same 3-simplex $\lift{\sigma}$ of $star(\lift{A}, \beta_r \lift{K})$. But as $\sigma\setminus \del A$ is contractible so $pr: \lift{\sigma} \setminus \del \lift{A} \to \sigma\setminus \del A$ is injective and we have a contradiction.

\end{proof}

The following result follows from Lemma 3.4 of \cite{KalPha}. We make minor modifications to work with ideal (possibly non-simplicial) triangulations in dimension 3, so we give a complete proof:

\begin{lemma}\label{Lem3.5}
Let $K$ be a geometric ideal $\theta_0$-thick triangulation of $M$. Let $\alpha K$ be a geometric subdivision of $K$ such that for all $3$-simplexes $A$ of $K$, $\alpha A$ is shellable. Let $s_i$ be the number of $i$-simplexes of $\alpha K$ that lie in the $i$-skeleton of $K$. Then $\alpha K$ is related to the derived subdivision $\beta K$ by a sequence of at most $(4\pi/\theta_0) s_1 + 2s_2 + s_3$ Pachner moves.
\end{lemma}
\begin{proof}
We shall obtain a sequence of Pachner moves $\alpha K=\beta_3^\alpha K \sim \beta_2^\alpha K \sim \beta_1^\alpha K \sim \beta_0^\alpha K = \beta K$.
Each step of this relation involves changing a subdivision of the star of an $r$-simplex $\delta$ of $\beta_r K$ to a cone over its boundary by Pachner moves. By a Pachner move in $star(\delta, \beta_r K)$ we in fact mean a Pachner move in the ball $star(\lift{\delta}, \beta_r \lift{K})$ which we then project down to $\beta_r K$. By Lemma \ref{starballs}, the projection from the interior of $star(\lift{\delta}, \beta_r \lift{K})$ to the interior of $star(\delta, \beta_r K)$ is an isometry and these Pachner moves do not change the boundary of $star(\lift{\delta}, \beta_r \lift{K})$. 

\emph{Step (i)} Let $A$ be a $3$-simplex of $K$. Fix a lift $\lift{A}$ of $A$ in $\lift{K}$. The subdivision of $A$, $\alpha A$, is given to be shellable, so $\alpha \lift{A}$ is shellable and by Lemma \ref{shellable} there exists a sequence of Pachner moves that changes $\alpha \lift{A}$ to $\lift{a} \star \del \alpha \lift{A}$ for a coning point $\lift{a}$ in the interior of $\lift{A}$. The covering projection $pr$ restricts to an isometry from $int(\lift{A})$ to $int(A)$ so we get a sequence of Pachner moves that star $\alpha A$.

Performing this starring operation on all $3$-simplexes $A$ of $K$ requires $s_3$ many Pachner moves. These moves change $\alpha K$ to $\beta^\alpha_2 K$, i.e., the $2$-skeleton of $K$ remains $\alpha$-subdivided while the $\alpha$-subdivision of the $3$-simplexes of $K$ is replaced by cones over their boundaries. Note that while the intermediate triangulations in this sequence are allowed to be non-geometric, $\beta_2^\alpha K$ is geometric again.

\emph{Step (ii)} Let $B$ be a 2-simplex of $K$ and let $star(B, K)=A_1 \cup A_2$, for $3$-simplexes $A_1$ and $A_2$ of $K$ (which may be equal). Let $\lift{B}$ be a lift of $B$ to $\lift{K}$ and let $star(\lift{B}, \lift{K})= \lift{A}_1 \cup \lift{A}_2$ which are distinct tetrahedra. If $a_1$ and $a_2$ are the coning points in the interior of $A_1$ and $A_2$ from Step (i), then $lk(\lift{B}, \beta_2 \lift{K})=\lift{a}_1 \cup \lift{a}_2$ where $\lift{a}_i$ is an interior point of $\lift{A}_i$.  The subdivision of any $2$-polytope is shellable so $\alpha B \star (a_1 \cup a_2)$ is also shellable. Applying Lemma \ref{shellable} again, we get a sequence of Pachner moves that change $\alpha \lift{B} \star (\lift{a}_1 \cup \lift{a}_2)$ to $\lift{b}\star \del \alpha \lift{B} \star (\lift{a}_1 \cup \lift{a}_2)$ for a point $\lift{b}$ in $int(\lift{B})$. Note that $\alpha \lift{B} \star lk(\lift{B}, \beta_2 \lift{K})$ is a subdivision of $star(\lift{B}, \beta_2 \lift{K})=\lift{B} \star lk(\lift{B}, \beta_2 \lift{K})$. By Lemma \ref{starballs} the projection map restricts to an isometry from $instar(\lift{B}, \beta_2 \lift{K}) \to instar(B, \beta_2 K)$. So we get a sequence of Pachner moves that changes $\alpha B \star lk(B, \beta_2 K)$ to $b \star \del \alpha B \star lk(B, \beta_2 K)$.

There are in total $2s_2$ many 3-simplexes in the union of all such $\alpha B \star (a_1 \cup a_2)$, so performing this starring operation on all $2$-simplexes $B$ of $K$ requires $2s_2$ many Pachner moves. These moves change $\beta^\alpha_2 K$ to $\beta^\alpha_1 K$, i.e., the $1$-skeleton of $K$ remains $\alpha$-subdivided while the $\alpha$-subdivisions of the $2$ and $3$ simplexes of $K$ in $\alpha K$ become cones over their boundaries. 

\emph{Step (iii)} Let $C$ be an edge of $K$ and let $star(\lift{C}, \lift{K})= \cup_{i=1}^n \lift{A}_i$ for $3$-simplexes $\lift{A}_i$ of $\lift{K}$ such that $\lift{B}_i=\lift{A}_i \cap \lift{A}_{i+1}$ is a $2$-simplex of $\lift{K}$.  If $\lift{a}_i\in int(\lift{A}_i)$ and $\lift{b}_i \in int(\lift{B}_i)$ are lifts of the corresponding coning points from Step (i) and Step (ii)  then the link of $\lift{C}$ in $\beta_1 \lift{K}$ is the circuit $(\lift{a}_1,\lift{b}_1, \lift{a}_2, \lift{b}_2, ..., \lift{a}_{n-1}, \lift{b}_{n-1}, \lift{a}_1)$ in the $1$-skeleton of $\beta_1 \lift{K}$. The join of shellable complexes is shellable so $\alpha \lift{C} \star lk(\lift{C}, \beta_1 \lift{K})$ is shellable as well. We proceed as before, starring this shellable complex using Lemma \ref{shellable} to change it to $\lift{c}\star \del \lift{C} \star lk(\lift{C}, \beta_1 \lift{K})$ for $\lift{c}$ an interior point of $\lift{C}$. Let $deg(\lift{C})$ denote the number of $3$-simplexes in $star(\lift{C}, \beta_1 \lift{K})$ and let $d=max(deg(\lift{C}))$ where the maximum is taken over all edges of $\lift{K}$. The lift $\lift{K}$ is $\theta_0$-thick, so the number of 3-simplexes $n$ in $star(\lift{C}, \lift{K})$ is at most $2\pi/\theta_0$. Therefore the number of edges in the circuit $lk(\lift{C}, \beta_1 \lift{K})=2n\leq 4\pi/\theta_0$, i.e., $d \leq 4\pi/\theta_0$. Note that $\alpha \lift{C} \star lk(\lift{C}, \beta_1 \lift{K})$ is a subdivision of $star(\lift{C}, \beta_1\lift{K})=\lift{C} \star lk(\lift{C}, \beta_1 \lift{K})$. By Lemma \ref{starballs} the projection map restricts to an isometry from $instar(\lift{C}, \beta_1 \lift{K}) \to instar(C, \beta_1 K)$. So we get a sequence of Pachner moves that changes $\alpha C \star lk(C, \beta_1 K)$ to $c \star \del \alpha C \star lk(C, \beta_1 K)$.

This starring operation for all edges $C$ of $K$ involves at most $ds_1 \leq (4\pi/\theta_0)s_1$ Pachner moves. And so making the corresponding Pachner moves for all edges $C$ of $K$ changes $\beta^\alpha_1 K$ to $\beta K$ in at most $s_1 (4\pi/\theta_0)$ moves.

We therefore obtain a sequence of Pachner moves that transform $\alpha K$ to $\beta K$ with length bounded by $s_3 + 2s_2 + (4\pi/\theta_0)s_1$ as required.
\end{proof}

\begin{lemma}\label{Pachnerlem} 
Let $K$ be a geometric ideal $\theta_0$-thick triangulation of $M$. Let $K'$ be a (possibly non-ideal) geometric subdivision of $K$. Let $p_i$ be the number of $i$-simplexes of $K$ for $i>0$. Let $s_i$ be the number of $i$-simplexes of $K'$ that lie in the $i$-skeleton of $K$. Then $\beta K'$ is related to $K$ by a sequence of at most $(8\pi/\theta_0) s_1 + 12s_2 + 24s_3 + (4\pi/\theta_0) p_1 + 2p_2 + p_3$ Pachner moves.
\end{lemma}
\begin{proof} 
We first bound the number of Pachner moves needed to transform $\beta K'$ to $\beta K$. Each $i$-simplex of $K'$ is split into $(i+1)!$ many $i$-simplexes on taking a derived subdivision.  The number of $i$-simplexes of $\beta K'$ in the $i$-skeleton of $K$ is therefore $(i+1)!s_i$. Denote $K'$ by $\alpha K$. Let $A$ be a simplex of $K$ and let $\alpha A$ be its $\alpha$-subdivision. Let $\alpha\lift{A}$ denote the lift of $\alpha A$ to the subdivision of the ideal simplex $\lift{A}$ in the Klein model of $\H^3$. Geodesics are straight lines in the Klein model so $\alpha\lift{A}$ is the subdivision of a Euclidean $3$-simplex in $\E^3$. By Theorem A of \cite{AdiBen}, its derived subdivision $\beta \alpha \lift{A}=\lift{\beta\alpha A}$ is shellable. Therefore  $(\beta \alpha)A$ is shellable for all $3$-simplexes $A$ of $K$. So replacing $s_i$ in Lemma \ref{Lem3.5} with $(i+1)!s_i$ we get the bound $(4\pi/\theta_0) (2s_1) + 2(6s_2) + (24s_3)$ on the number of Pachner moves needed to go from $\beta K'$ to $\beta K$.

A single 3-simplex is trivially shellable, so we next take $\alpha K = K$ and $s_i=p_i$ in Lemma \ref{Lem3.5}. This gives the bound $(4\pi/\theta_0) p_1 + 2p_2 + p_3$ on the number of Pachner moves needed to relate $K$ and $\beta K$.

Putting these sequences of Pachner moves together we get the required bound on the number of Pachner moves needed to go from $\beta K'$ to $K$.
\end{proof}

\begin{lemma}\label{polycount}
Let $K_1$ and $K_2$ be geometric ideal $\theta_0$-thick triangulations of $M$. Let $s$ be the number of $3$-simplexes of $K'=\beta(K_1 \cap K_2)$ and let $f$ be the number of $3$-polytopes in $K_1 \cap K_2$. Then $s \leq 112f$.
\end{lemma}
\begin{proof}
Let $\Delta_1$ and $\Delta_2$ be 3-simplexes of $K_1$ and $K_2$ respectively and let $P$ be a 3-polytopal component of $\Delta_1 \cap \Delta_2$. Let $F(P)$ denote the number of faces of $P$. Each face of $P$ is a subset of a unique face of $\Delta_1$ or $\Delta_2$ so $F(P) \leq 8$. Consequently each face has at most 7 edges and on taking derived subdivisions, each face splits into at most $14$ $2$-simplexes. Therefore $\beta P$ has at most $14 F(P) \leq 112$ many $3$-simplexes. Summing over all $3$-polytopes $P$ of $K_1 \cap K_2$ gives $s \leq 112 f$.
\end{proof}

We can now finally calculate an explicit bound on the number of Pachner moves needed to relate geometric ideal $\theta_0$-thick triangulations:
\begin{lemma}\label{mainlem1}
Let $M$ be a complete orientable cusped hyperbolic $3$-manifold. Let $\tau_1$ and $\tau_2$ be geometric ideal triangulations of $M$ with at most $m_1$ and $m_2$ many $3$-simplexes respectively and all dihedral angles at least  $\theta_0$. Let $m=m_1 + m_2$. Then the number of Pachner moves needed to relate $\tau_1$ and $\tau_2$ is less than 
$$N(m, \theta)=(10752+3584\pi/\theta_0)f + (5+8\pi/\theta_0)m$$ where,
\begin{align*}
f=&\left( \frac{4\pi v_{tet}}{\theta_0^2 (sinh(r_0) - r_0)}+1\right)m\\
r_0=& \arcsinh(\sinh(a_0/2)\sin\theta_0)\\
a_0=&\arcsinh(l_0\sin\theta_0/z_0)\\
z_0=&\sqrt{2m\,v_{tet}\cot\theta_0}/(\epsilon\sin\theta_0)\\
l_0=& (\sin\theta_0)^{4m}(\sqrt{m^2+2m}-m)/(4m)
\end{align*}
\end{lemma}

\begin{proof}
Let $p_i$ and $q_i$ be the number of $i$-simplexes of $\tau_1$ and $\tau_2$ respectively. The polytopal complex $\tau_1 \cap \tau_2$ has $f$ many $3$-polytopes as given by Theorem \ref{mainsubthm}. Let $s_i$ be the the number of $i$-simplexes of $K'=\beta(\tau_1 \cap \tau_2)$ which lie in the $i$-skeleton of $\tau_1 \cap \tau_2$. So $s_i$ is greater than or equal to the number of $i$-simplexes of $K'$ which lie in the $i$-skeleton of $\tau_1$ and of $\tau_2$. Applying Lemma \ref{Pachnerlem} twice, we get a bound on the number of Pachner moves to relate $\tau_1$ and $\tau_2$ via $\beta K'$ as $(16\pi/\theta_0) s_1 + 24s_2 + 48s_3 + (4\pi/\theta_0) (p_1+q_1) + 2(p_2+q_2) + (p_3+q_3)$.

Each face of $\lift{\tau}_1$ lies in 2 tetrahedra and each tetrahedron has 4 faces so $4p_3$ counts each face of $\tau_1$ exactly twice, therefore $2p_2 =4p_3$. Similarly, as each edge of $\lift{\tau}_1$ lies in at least $3$ tetrahedra and each tetrahedron has $6$ edges so $3p_1 \leq 6p_3$. As $p_3=m_1$ so we get $p_1 \leq 2m_1$ and $p_2 = 2m_1$. Similar identities hold for $q_i$. And similarly, each face of $\lift{K}'$ lies in two tetrahedra of $\lift{K}'$ but some face of $K'$ may not lie in the 2-skeleton of $\tau_1 \cap \tau_2$ so we get $2s_2 \leq 4s_3$. Each edge of $\lift{K}'$ lies in at least $3$ tetrahedron of $\lift{K}'$ and as each tetrahedron of $K'$ has at most $6$ edges which lie in the $1$-skeleton of $\tau_1 \cap \tau_2$ so $3s_1 \leq 6s_3$.

Plugging in these values into the bound above gives us the bound
$((16\pi/\theta_0)2 +(24)2 + 48)s + ((4\pi/\theta_0)2 + (2)2 + 1)(m_1+m_2)$. Using the inequality $s \leq 112f$ obtained in Lemma \ref{polycount}, we get the bound $(10752+3584\pi/\theta_0)f + (5+8\pi/\theta_0)m$. Finally, we can plug in the value for $f$ from Theorem \ref{mainsubthm} to get the required bound.

\end{proof}

This bound in the above lemma can be simplified to highlight the dependence on $\theta$ and $m$ which proves the main result of this article:
\begin{proof}[Proof of Theorem \ref{mainthm1}]
This proof is just a simplification of the bound obtained in Lemma \ref{mainlem1}. By Lemma \ref{abound} and as $\sinh$ is an increasing function so,
$$\sinh(a_0/2) \sin\theta_0<\sinh(a_0)<1$$
From the calculations in proof of Theorem \ref{mainthm2} we get,
$$\sinh \left(\frac{a_0}{2}\right) \sin \theta_0 > \frac{\sqrt{3}}{4}\cdot \frac{\epsilon(\sin\theta_0)^{4m+5/2}}{4\sqrt{2mv_{tet}}}\cdot \frac{1}{2m}\cdot\sin\theta_0$$

As $t- \arcsinh(t) \geq \frac{t^3}{6} - \frac{3t^5}{40} \geq \frac{t^3}{12}$ for $t\leq 1$ so
putting $t=\sinh(a_0/2)\sin\theta_0$ and substituting the value $r_0=\arcsinh(\sinh(a_0/2)\sin\theta_0)$ we get

\begin{align*}
\sinh(r_0) - r_0 &= \left(\sinh \left(\frac{a_0}{2}\right) \sin \theta_0 - \arcsinh\left(\sinh \left(\frac{a_0}{2}\right) \sin \theta_0\right)\right)\\
    &\geq \left(\frac{\left(\sinh \left(\frac{a_0}{2}\right) \sin \theta_0\right)^3}{12} \right)
     \geq \frac{3 \sqrt{3} \, \epsilon^3}{2^{16}\sqrt{2}\cdot 12 \cdot  (v_{tet})^{3/2}}\frac{(\sin \theta_0)^{12m+21/2}}{m^{9/2}}\\
\end{align*}
This gives,
\begin{align*}
\frac{4\pi v_{tet}}{\theta_0^2 \, (\sinh(r_0)-r_0)} &\leq \frac{2^{18} \cdot 12  \sqrt{2}\pi (v_{tet})^{5/2}}{3\sqrt{3}\epsilon^3}\frac{m^{9/2}}{\theta_0^2 \, (\sin \theta_0)^{12m+21/2}}\\
     \frac{4\pi v_{tet}}{\theta_0^2 \, (\sinh(r_0)-r_0)} +1 &\leq \frac{2^{18} \cdot 12  \sqrt{2}\pi (v_{tet})^{5/2}m^{9/2}+3\sqrt{3}\epsilon^3\,\theta_0^2 \, (\sin \theta_0)^{12m+21/2}}{3\sqrt{3}\epsilon^3\,\theta_0^2 \, (\sin \theta_0)^{12m+21/2}}\\
\end{align*}
As $\theta_0 \leq \pi/3$, $\epsilon < 1$, $1<v_{tet}$ and $4\leq m$ so 
$$
3\sqrt{3}\epsilon^3 \theta_0^2 (\sin\theta_0)^{12m+21/2}<2^{18}\sqrt{2}\pi(4)^{9/2}<2^{18}\sqrt{2}\pi(v_{tet})^{5/2}m^{9/2}
$$
As $\epsilon\geq 0.29$ and $v_{tet} < 1.015$, so we have,
\begin{align*}
f(m, \theta_0)&=\left( \frac{4\pi v_{tet}}{\theta_0^2 \, (\sinh(r_0)-r_0)}+1\right)m \leq \frac{2^{18}\cdot 13\sqrt{2} \pi (v_{tet})^{5/2}m^{11/2}}{3\sqrt{3}\epsilon^3\,\theta_0^2 \, (\sin \theta_0)^{12m+21/2}}\\
&< \frac{2^{18}\sqrt{2}(335) m^{11/2}}{\theta_0^2 \, (\sin \theta_0)^{12m+21/2}}
\end{align*}
Let $N(m, \theta_0)$ be the number of Pachner moves required to relate $\tau_1$ and $\tau_2$.
\begin{align*}
    \text{So } N(m, \theta) &\leq \left(10752+\frac{3584\pi}{\theta_0}\right)f + \left(5+ \frac{8\pi}{\theta_0}\right)m\\
    &\leq \left(10752k+5 \right)m + \frac{(3584k+8)\pi}{\theta_0} m\\
    \text{ where } k&= \frac{2^{18}\sqrt{2} (335) m^{9/2}}{\theta_0^2 \, (\sin \theta_0)^{12m+21/2}} 
\end{align*}

\begin{align*}
10752k+5 &=  \frac{10752(335) \, 2^{18}\sqrt{2} \, m^{9/2}}{\theta_0^2 \, (\sin \theta_0)^{12m+21/2}} +5\\
    &=  \frac{10752(335) \, 2^{18}\sqrt{2} \, m^{9/2} + 5 \theta_0^2 \, (\sin \theta_0)^{12m+21/2}}{\theta_0^2 \, (\sin \theta_0)^{12m+21/2}}\\
    &\leq \frac{(1.33534\times10^{12}) \, m^{9/2}}{\theta_0^2 \, (\sin \theta_0)^{12m+21/2}} \text{  as $5\theta_0^2 (\sin\theta_0)^{12m+21/2} < 4^{9/2} \leq m^{9/2}$}
\end{align*}
Similarly,
\begin{align*}
   3584k+8 &\leq  \frac{(4.45111\times 10^{11}) \, m^{9/2}}{\theta_0^2 \, (\sin \theta_0)^{12m+21/2}} \text{ as $ 8\theta_0^2 (\sin\theta_0)^{12m+21/2}< 4^{9/2}\leq m^{9/2}$}
\end{align*}
As $1/t<1/\sin(t)$ for $t>0$ and as $\theta_0\leq \pi/3$, so from Lemma \ref{mainlem1} we get
\begin{align*}
N(m, \theta)
&\leq \frac{ m^{11/2}}{\theta_0^2 \, (\sin \theta_0)^{12m+21/2}}\left((1.33534\times10^{12})+\frac{(4.45111\times 10^{11})\pi}{\theta_0}\right)\\
&\leq \frac{ m^{11/2}}{(\sin\theta_0)^3 \, (\sin \theta_0)^{12m+21/2}}\left((1.33534\times10^{12})\left(\frac{\pi}{3}\right)+(4.45111\times 10^{11})\pi\right)\\
N(m, \theta) &< (2.797\times 10^{12})\frac{ m^{11/2}}{(\sin \theta_0)^{12m+27/2}} 
\end{align*}
\end{proof}

\begin{acknowledgements}
The first author was supported by the MATRICS grant of Science and Engineering Research Board, GoI. We would like to thank the referees for comments that have significantly improved the exposition.
\end{acknowledgements}

\bibliographystyle{alpha}
\bibliography{Revised}

\begin{thebibliography}{HTW98}

\bibitem[AB17]{AdiBen}
Karim~A. Adiprasito and Bruno Benedetti.
\newblock Subdivisions, shellability, and collapsibility of products.
\newblock {\em Combinatorica}, 37(1):1--30, 2017.

\bibitem[Ada02]{Ada2002}
Colin~C. Adams.
\newblock Waist size for cusps in hyperbolic 3-manifolds.
\newblock {\em Topology}, 41(2):257--270, 2002.

\bibitem[AHS99]{AdaHasSco}
Colin Adams, Joel Hass, and Peter Scott.
\newblock Simple closed geodesics in hyperbolic {$3$}-manifolds.
\newblock {\em Bull. London Math. Soc.}, 31(1):81--86, 1999.

\bibitem[Ame05]{Ame}
Gennaro Amendola.
\newblock A calculus for ideal triangulations of three-manifolds with embedded
  arcs.
\newblock {\em Math. Nachr.}, 278(9):975--994, 2005.

\bibitem[AR00]{AdaRei}
Colin~C. Adams and Alan~W. Reid.
\newblock Systoles of hyperbolic {$3$}-manifolds.
\newblock {\em Math. Proc. Cambridge Philos. Soc.}, 128(1):103--110, 2000.

\bibitem[BP92]{BenPet}
Riccardo Benedetti and Carlo Petronio.
\newblock {\em Lectures on hyperbolic geometry}.
\newblock Universitext. Springer-Verlag, Berlin, 1992.

\bibitem[Bre09]{Bre}
William Breslin.
\newblock Thick triangulations of hyperbolic {$n$}-manifolds.
\newblock {\em Pacific J. Math.}, 241(2):215--225, 2009.

\bibitem[Bur14]{Bur}
Benjamin~A. Burton.
\newblock A duplicate pair in the {S}nap{P}ea census.
\newblock {\em Exp. Math.}, 23(2):170--173, 2014.

\bibitem[CL14]{CowLac}
Alexander Coward and Marc Lackenby.
\newblock An upper bound on {R}eidemeister moves.
\newblock {\em Amer. J. Math.}, 136(4):1023--1066, 2014.

\bibitem[CM01]{CaoMey}
Chun Cao and G.~Robert Meyerhoff.
\newblock The orientable cusped hyperbolic {$3$}-manifolds of minimum volume.
\newblock {\em Invent. Math.}, 146(3):451--478, 2001.

\bibitem[DD16]{DadDua}
Blake Dadd and Aochen Duan.
\newblock Constructing infinitely many geometric triangulations of the figure
  eight knot complement.
\newblock {\em Proc. Amer. Math. Soc.}, 144(10):4545--4555, 2016.

\bibitem[EP88]{EpsPen}
D.~B.~A. Epstein and R.~C. Penner.
\newblock Euclidean decompositions of noncompact hyperbolic manifolds.
\newblock {\em J. Differential Geom.}, 27(1):67--80, 1988.

\bibitem[FP07]{FutPur}
David Futer and Jessica~S. Purcell.
\newblock Links with no exceptional surgeries.
\newblock {\em Comment. Math. Helv.}, 82(3):629--664, 2007.

\bibitem[GL89]{GorLue}
C.~McA. Gordon and J.~Luecke.
\newblock Knots are determined by their complements.
\newblock {\em J. Amer. Math. Soc.}, 2(2):371--415, 1989.

\bibitem[Hak68]{Hak}
Wolfgang Haken.
\newblock Some results on surfaces in {$3$}-manifolds.
\newblock In {\em Studies in {M}odern {T}opology}, pages 39--98. Math. Assoc.
  Amer. (distributed by Prentice-Hall, Englewood Cliffs, N.J.), 1968.

\bibitem[Hem79]{Hem}
Geoffrey Hemion.
\newblock On the classification of homeomorphisms of {$2$}-manifolds and the
  classification of {$3$}-manifolds.
\newblock {\em Acta Math.}, 142(1-2):123--155, 1979.

\bibitem[HP20]{HamPur}
Sophie~L. {Ham} and Jessica~S. {Purcell}.
\newblock {Geometric triangulations and highly twisted links}.
\newblock {\em arXiv e-prints}, page arXiv:2005.11899, May 2020.

\bibitem[HTW98]{HosThiWee}
Jim Hoste, Morwen Thistlethwaite, and Jeff Weeks.
\newblock The first 1,701,936 knots.
\newblock {\em Math. Intelligencer}, 20(4):33--48, 1998.

\bibitem[IS10]{IzmSch}
Ivan Izmestiev and Jean-Marc Schlenker.
\newblock Infinitesimal rigidity of polyhedra with vertices in convex position.
\newblock {\em Pacific J. Math.}, 248(1):171--190, 2010.

\bibitem[KP20]{KalPha2}
Tejas Kalelkar and Advait Phanse.
\newblock Geometric bistellar moves relate geometric triangulations.
\newblock {\em Topology Appl.}, 285:107390, 7, 2020.

\bibitem[KP21]{KalPha}
Tejas Kalelkar and Advait Phanse.
\newblock An upper bound on pachner moves relating geometric triangulations.
\newblock {\em Discrete {\&} Computational Geometry}, Mar 2021.

\bibitem[Kup19]{Kup}
Greg Kuperberg.
\newblock Algorithmic homeomorphism of 3-manifolds as a corollary of
  geometrization.
\newblock {\em Pacific J. Math.}, 301(1):189--241, 2019.

\bibitem[Lac04]{Lac}
Marc Lackenby.
\newblock The volume of hyperbolic alternating link complements.
\newblock {\em Proc. London Math. Soc. (3)}, 88(1):204--224, 2004.
\newblock With an appendix by Ian Agol and Dylan Thurston.

\bibitem[Lic91]{Lic2}
W.~B.~R. Lickorish.
\newblock Unshellable triangulations of spheres.
\newblock {\em European J. Combin.}, 12(6):527--530, 1991.

\bibitem[Lic99]{Lic}
W.~B.~R. Lickorish.
\newblock Simplicial moves on complexes and manifolds.
\newblock In {\em Proceedings of the {K}irbyfest ({B}erkeley, {CA}, 1998)},
  volume~2 of {\em Geom. Topol. Monogr.}, pages 299--320. Geom. Topol. Publ.,
  Coventry, 1999.

\bibitem[LST08]{LuoSchTil}
Feng Luo, Saul Schleimer, and Stephan Tillmann.
\newblock Geodesic ideal triangulations exist virtually.
\newblock {\em Proc. Amer. Math. Soc.}, 136(7):2625--2630, 2008.

\bibitem[Man02]{Man}
Jason Manning.
\newblock Algorithmic detection and description of hyperbolic structures on
  closed 3-manifolds with solvable word problem.
\newblock {\em Geom. Topol.}, 6:1--25, 2002.

\bibitem[Mat07]{Mat}
Sergei Matveev.
\newblock {\em Algorithmic topology and classification of 3-manifolds},
  volume~9 of {\em Algorithms and Computation in Mathematics}.
\newblock Springer, Berlin, second edition, 2007.

\bibitem[Men84]{Men}
W.~Menasco.
\newblock Closed incompressible surfaces in alternating knot and link
  complements.
\newblock {\em Topology}, 23(1):37--44, 1984.

\bibitem[Mij05]{Mij4}
Aleksandar Mijatovi\'{c}.
\newblock Simplical structures of knot complements.
\newblock {\em Math. Res. Lett.}, 12(5-6):843--856, 2005.

\bibitem[Mos73]{Mos}
G.~D. Mostow.
\newblock {\em Strong rigidity of locally symmetric spaces}.
\newblock Princeton University Press, Princeton, N.J.; University of Tokyo
  Press, Tokyo, 1973.
\newblock Annals of Mathematics Studies, No. 78.

\bibitem[Pra73]{Pra}
Gopal Prasad.
\newblock Strong rigidity of {${\bf Q}$}-rank {$1$} lattices.
\newblock {\em Invent. Math.}, 21:255--286, 1973.

\bibitem[Rud58]{Rud}
Mary~Ellen Rudin.
\newblock An unshellable triangulation of a tetrahedron.
\newblock {\em Bull. Amer. Math. Soc.}, 64:90--91, 1958.

\bibitem[Sel95]{Sel}
Z.~Sela.
\newblock The isomorphism problem for hyperbolic groups. {I}.
\newblock {\em Ann. of Math. (2)}, 141(2):217--283, 1995.

\bibitem[Sha11]{Sha}
Peter~B. Shalen.
\newblock A generic {M}argulis number for hyperbolic 3-manifolds.
\newblock In {\em Topology and geometry in dimension three}, volume 560 of {\em
  Contemp. Math.}, pages 103--109. Amer. Math. Soc., Providence, RI, 2011.

\bibitem[Wee93]{Wee}
Jeffrey~R. Weeks.
\newblock Convex hulls and isometries of cusped hyperbolic {$3$}-manifolds.
\newblock {\em Topology Appl.}, 52(2):127--149, 1993.

\end{thebibliography}

\end{document}